\documentclass[a4paper,oneside,portrait,11pt]{AmsArt}

\usepackage[margin=1in]{geometry}

\usepackage{amsfonts,amscd,amsthm,amsgen,amsmath,amssymb}
\usepackage[all]{xy}
\usepackage[vcentermath]{youngtab}
\usepackage {graphicx}
\usepackage{epstopdf}
\usepackage{amsfonts}
\usepackage{amsthm}
\usepackage{amsmath}
\usepackage{amsfonts}
\usepackage{latexsym}
\usepackage{amssymb}
\usepackage{epsfig,color}
\usepackage{graphicx, amssymb}
 \usepackage{color}
 \usepackage{epsfig,color}
\usepackage{graphicx}
\usepackage{caption}
\usepackage{subcaption}
\usepackage{wrapfig}

\newtheorem{rem}{Remark}[] 

\newtheorem{prop}[rem]{Proposition}
\newtheorem{thm}[rem]{Theorem}
\newtheorem{lem}[rem]{Lemma}

\newtheorem{proposition}[rem]{Proposition}
\newtheorem{theorem}[rem]{Theorem}

\newtheorem{ex}[rem]{Example}
\newtheorem{remark}[rem]{Remark}
\newtheorem{example}[rem]{Example}

\theoremstyle{definition}

\newtheorem{definition}[rem]{Definition}
\newtheorem{defn}[rem]{Definition}

\def\CA{{\mathcal A}}

\def\CL{{\mathcal L}}
\def\CM{{\mathcal M}}

\newcommand{\C}{\mathbb{C}}

\newcommand{\R}{\mathbb{R}}

\title{Higgs bundles and $(A,B,A)$-branes}
 
\author{David Baraglia and Laura P. Schaposnik}

\address{School of Mathematical Sciences, The University of Adelaide, SA 5005 Australia.}
\email{david.baraglia@adelaide.edu.au}
\address{ Mathematisches Institut, Ruprecht-Karls-Universit\"{a}t Heidelberg, 69120 Heidelberg, Germany.}
\email{schaposnik@mathi.uni-heidelberg.de }

\date{\today}

\begin{document}
\maketitle

\begin{abstract}
Through the action of anti-holomorphic involutions on a compact Riemann surface $\Sigma$, we construct families of $(A,B,A)$-branes $\CL_{G_{c}}$ in the moduli spaces $\mathcal{M}_{G_{c}}$ of $G_{c}$-Higgs bundles on $\Sigma$. We study the geometry of these $(A,B,A)$-branes in terms of spectral data and show they have the structure of real integrable systems.
\end{abstract}


\setcounter{tocdepth}{1}

\tableofcontents


\section{Introduction}
 
 Since Higgs bundles were introduced in 1987 \cite{N1}, they have found applications in many areas of mathematics and mathematical physics. In particular, Hitchin showed in  \cite{N1} that their moduli spaces give examples of hyperk\"ahler manifolds and that they provide interesting examples of integrable systems \cite{N2}. More recently,  Hausel and Thaddeus \cite{Tamas1}  related Higgs bundles to mirror symmetry, and in the work of Kapustin and Witten \cite{Kap} Higgs bundles were used to give a physical derivation of
the geometric Langlands correspondence.

Classically, a Higgs bundle $(E,\Phi)$ on a compact Riemann surface  $\Sigma$ of genus $g\geq 2$, is given by a holomorphic vector bundle $E$ on $\Sigma$ together with a holomorphic section $\Phi: E\rightarrow E \otimes K$, for  $K$ the canonical bundle of the surface. Moreover, given a complex semisimple Lie group $G_{c}$,  one can define $G_{c}$-Higgs bundles \cite{N2} as pairs $(P,\Phi)$ on $\Sigma$, for $P$ a principal $G_{c}$ bundle and the Higgs field $\Phi$  a holomorphic section of ${\rm ad}(P)\otimes K$, for ${\rm ad}(P)$ the adjoint bundle. By considering parabolic stability, one can construct the moduli space $\CM_{G_{c}}$ of $G_{c}$-Higgs bundles (for details, see \cite[Section 3]{biswas} and references therein), which carries a hyperk\"ahler structure.

Through \textcolor{black}{a choice of three complex structures giving } the hyperk\"ahler structure of the moduli spaces $\CM_{G_{c}}$, one can study submanifolds which are of type $A$ (Lagrangian) or $B$ (complex) with respect to each of the three hyperk\"ahler complex structures $(I,J,K)$, and which are sometimes referred to as branes. We dedicate this paper to the study of branes of type $(A,B,A)$ arising naturally from anti-holomorphic involutions acting on the moduli space of $G_{c}$-Higgs bundles. These are submanifolds which are Lagrangian with respect to the two symplectic structures $\omega_I,\omega_K$ and complex with respect to $J$. Such branes are of considerable interest due to the connections believed to exist between Langlands duality and homological mirror symmetry. Additionally the $(A,B,A)$-branes we construct turn out to be integrable systems making them interesting structures to consider in their own right.

Given an anti-holomorphic involution $f$ on a compact Riemann surface $\Sigma$ of genus $g\geq 2$, one has a natural induced action on the moduli space of $G_{c}$-Higgs bundles as well as on the moduli space of representations of $\pi_{1}(\Sigma)$ into $G_{c}$, which we study in Sections \ref{sec:anti-inv}-\ref{sec:inro-Higgs}, and whose fixed point set we denote by $\CL_{G_{c}}$. By looking at the fixed points of this involution, in Section \ref{sec:aba} we obtain natural $(A,B,A)$-branes, in the sense of Kapustin and Witten \cite{Kap}:

\vspace{0.1 in}
 
\noindent {\bf Theorem \ref{teo1}.} For each choice of anti-holomorphic involution $f$ on a compact Riemann surface, there is a natural $(A,B,A)$-brane $\CL_{G_{c}}$ defined in the moduli space of $G_{c}$-Higgs bundles. 

\vspace{0.1 in}

In order to study these branes, in Sections \ref{sec:hitFib}-\ref{sec:invo-spec} we consider the spectral data associated to $G_{c}$-Higgs bundles introduced in \cite{N2}, for $G_{c}$ a complex classical Lie group, and look at   $\CL_{G_{c}}$ as sitting inside the corresponding Hitchin fibration, obtaining new examples of real integrable systems \textcolor{black}{over some subset $L$ of the base of the  Hitchin fibration}:

 \vspace{0.1 in}

 \noindent {\bf Theorem \ref{thm:intsys}.}
If $\mathcal{L}_{G_c}$ contains smooth points then the restriction of the Hitchin fibration $h|_{\mathcal{L}_{G_c}} : \mathcal{L}_{G_c} \to L$ to $\mathcal{L}_{G_c}$ is a Lagrangian fibration with singularities. The generic fibre is smooth and consists of a finite number of tori.

 \vspace{0.1 in}

Through the study of the spectral data associated to such $(A,B,A)$-branes in Section \ref{sec:invo-spec}, we prove  that $\mathcal{L}_{G_c}$ always contains smooth points for the classical Lie groups. Hence Theorem \ref{thm:intsys} above applies and we obtain many families of real integrable systems through this construction.\

In Section \ref{sec:conn} we study the connectivity of the fibres of $\CL_{G_{c}}$. In the $GL(n,\mathbb{C})$ case we find in Proposition \ref{prop:numbergl} the number of component in terms of the number of components of the fixed point set of an induced involution on the spectral curves. In the rank $2$ case we are able to reduce this to an exact formula in terms of the associated quadratic differential (Theorem \ref{thm:numbergl2}). We also obtain exact formulas in the $SL(2,\mathbb{C})$ case (Theorem \ref{thm:numbersl2b}).

For certain groups $G_c$ we find that the Higgs bundles fixed by the induced involution may have a real or quaternionic structure. In particular we consider this for $SL(2,\mathbb{C})$-bundles in Section \ref{sec:realquat}. We find that if $(E,\Phi)$ is a stable $SL(2,\mathbb{C})$-Higgs bundle then $E$ carries either a real or quaternionic structure. Moreover we determine which of the two possibilities occur in terms of the spectral data associated to $(E,\Phi)$.

Under Langlands duality, it is known (\cite[Section 12.4]{Kap}) that $(A,B,A)$-branes in $\CM_{G_{c}}$ map to $(A,B,A)$-branes in the moduli space $\CM_{^{L}G_{c}}$ of $^{L}G_{c}$-Higgs bundles, for $^{L}G_{c}$ the Langlands dual group of $G_{c}$. We propose that the dual of the $(A,B,A)$-brane $\CL_{G_c}$ is the correspoinding $(A,B,A)$-brane $\CL_{^{L}G_c}$ defined using the same involution $f$ on the Riemann surface. In Section \ref{sec:lang} we present some preliminary evidence in support this proposal.

Following the ideas of Kapustin-Witten \cite{Kap}, and Gukov \cite{gukov07}, the constructions given in this paper can be shown to be closely related to representations of 3-manifolds whose boundary is the Riemann surface $\Sigma$.      
In Section \ref{sec:3man}, we give a brief description of this relation, which is studied and developed in the companion paper \cite{BS13}.

\vspace{0.1 in}

 \noindent{\bf \ Acknowledgements:} The authors would like to thank S. Gukov and N. Hitchin for inspiring and helpful conversations, and to D. Alessandrini, J. Stix and A. Wienhard for useful comments.



\section{Anti-holomorphic involutions on a Riemann surface}\label{sec:anti-inv}

We consider anti-holomorphic involutions $f : \Sigma \to \Sigma$ on a compact Riemann surface $\Sigma$, always taken to be connected. Such an involution $f$ induces corresponding involutions on the moduli space of representations $\pi_1(\Sigma) \to G_c$ and the moduli space of $G_{c}$-Higgs bundles on $\Sigma$, for a given complex semisimple Lie group $G_c$. To construct these extensions, we shall begin by studying the different possible anti-holomorphic involutions $f$, and the induced actions on the fundamental group $\pi_{1}(\Sigma)$.

\subsection{Topological classification}\label{sec:classif}
The classification of anti-holomorphic involutions  of a compact Riemann surface $f : \Sigma \to \Sigma$ is a classical result of Klein. The reader should refer to \cite{GH81} for a through study of this situation. All such involutions on $\Sigma$ may be characterised by two integer invariants $(n,a)$ as follows. The fixed point set of $f$ is known to be a disjoint union of copies of the circle embedded in the surface. Let $n$ denote the number of components of the fixed point set, which by a classical theorem of Harnack, can be at most $g+1$. The complement of the fixed point set has one or two components. Set $a=0$ if the complement is disconnected and $a=1$ otherwise. 

\begin{example} For $\Sigma$ of genus $2$ and $n(\Sigma)=1$,  one may have:

\begin{figure}[h]
        \centering
 \begin{subfigure}[b]{0.3\textwidth}
                 \centering
  \epsfig{file=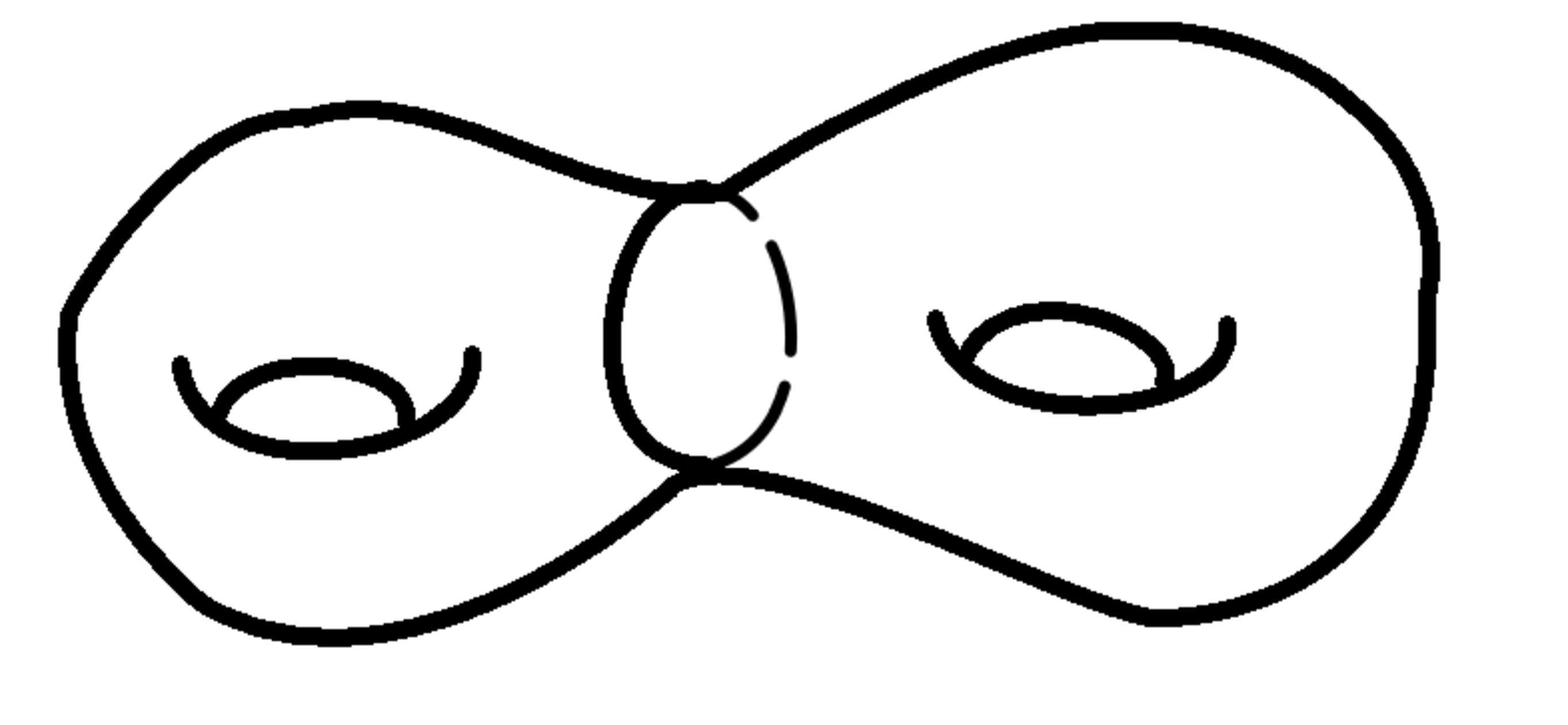, width=\textwidth}
  \caption{\begin{small}$ a(\Sigma)=0$\end{small}}\label{case1}
          \end{subfigure}
          \hspace{0.4 in}
                     \begin{subfigure}[b]{0.3\textwidth}
                             \centering
 \epsfig{file=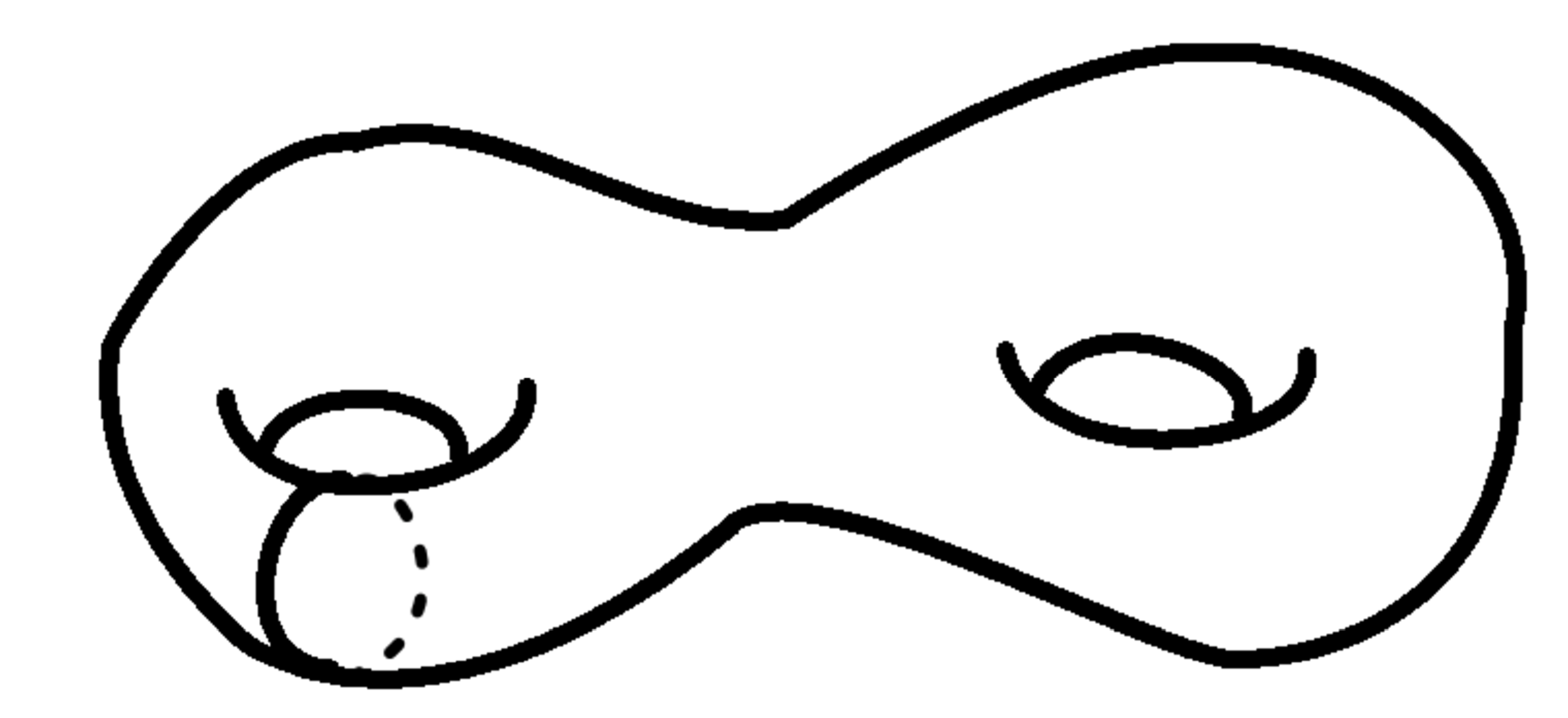, width=\textwidth}
   \caption{\begin{small}$ a(\Sigma)=1$\end{small}}\label{case3}
          \end{subfigure}
          %
\label{cases}
\end{figure}
\end{example}
 \vspace{-0.2 in}
  \begin{remark}
 One should note that an anti-involution $f$ as considered in this paper can also be found in the literature as an anti-conformal map of $\Sigma$, whose species $Spi(f)$ is $+k$ or $-k$ according to whether $\Sigma - {\rm Fix}(f)$ is connected or not (e.g. see \cite{acc2} and references therein).
 \end{remark}

\begin{prop}[\cite{GH81}]\label{pro:topcond}
Let $g$ be the genus of $\Sigma$. The invariants $(n,a)$ associated to an orientation reversing involution on $\Sigma$ satisfy the following conditions
\begin{itemize}
\item{$0 \le n \le g+1$.}
\item{If $n=0$ then $a=1$. If $n=g+1$ then $a=0$.}
\item{If $a=0$ then $n = g+1 \; ( {\rm mod} \; 2)$.}
\end{itemize}
Conversely any pair $(n,a)$ satisfying these conditions determines such an orientation reversing involution on $\Sigma$, unique up to homeomorphism.
\end{prop}

Orientation reversing involutions on a compact Riemann surface $\Sigma$ may be constructed as follows. Take a compact surface $S$ (orientable or non-orientable) with boundary, and let $\tilde{S}$ be its orientation double cover, which is an oriented surface with boundary. The involution $f$ which interchanges the sheets of the double cover reverses the orientation of $\tilde{S}$. To each boundary component of $S$ we obtain a corresponding pair of boundary components in $\tilde{S}$ which are exchanged by $f$. Then we can define $\Sigma$ by gluing together the boundary components of $\tilde{S}$ through the restriction of $f$ to the boundary. Moreover, such a surface $\Sigma$ has a natural orientation reversing involution induced by $f$. One finds that the conditions of Proposition \ref{pro:topcond} are satisfied by this construction.

In order to see that the above construction produces all such anti-holomorphic involutions on a compact Riemann surface, consider $\Sigma$ a compact Riemann surface with an anti-holomorphic involution $f$ which has fixed points. In a neighbourhood of a fixed point, one can find a local coordinate $z$ such that the involution $f$ is given by the complex conjugation $f:z\mapsto \bar z$ (e.g. see \cite{Se91} or \cite{GH81}). Hence, the fixed point set can be seen as a union of copies of the unit circle embedded in the surface. By making cuts in $\Sigma$  around these circles we obtain a Riemann surface $\Sigma'$ with two boundary components for each cut, and a fixed point free anti-holomorphic involution, which permutes each pair of components. Since the induced involution on $\Sigma '$ does not fix any boundary component, it pairs off components and thus $\Sigma'$ is the orientation double cover of a Riemann surface with boundary.



\subsection{Action on spin structures}

Recall that a spin structure on an oriented Riemannian $n$-manifold $M$ with $SO(n)$-frame bundle $P$ may be defined as a class $\xi \in H^1(P,\mathbb{Z}_2)$ which restricted to any fibre of $P \to M$ agrees with the class in $H^1(SO(n),\mathbb{Z}_2)$ corresponding to the double cover $Spin(n) \to SO(n)$. In the case that $M = \Sigma$ is a Riemann surface we may identify the frame bundle $P$ with the bundle $U\Sigma$ of unit tangent vectors, since any unit vector can be uniquely extended to an oriented frame.

Throughout the paper we let $K$ denote the canonical bundle of a Riemann surface $\Sigma$. Given a compatible Riemannian metric on $\Sigma$, the spin structures correspond to theta characteristics, holomorphic line bundles $L$ for which there is an isomorphism $L^2 \simeq K$. 
The hermitian structure on $K$ determines a hermitian structure on $L$ such that the bundle of unit vectors in $L$ is a double cover of the unit vectors of $K$. Moreover the unit vectors of $K$ can be naturally identified with the unit tangent bundle $U\Sigma$, so from $L$ we obtain a double cover of $U\Sigma$ which is then a spin structure. The spin structures on $\Sigma$ for different choices of metric may be canonically identified, and the spin structure associated to $L$ does not depend on the choice of metric, for different choices of metrics compatible with the complex structure on $\Sigma$.

For $f $ an anti-holomorphic involution on $\Sigma$, the induced map $f^* : H^1(U\Sigma , \mathbb{Z}_2) \to H^1(U\Sigma , \mathbb{Z}_2)$ preserves the subset of classes in $H^1(U\Sigma , \mathbb{Z}_2)$ which define spin structures, thus $f$ has a natural action on the set of spin structures. Moreover, from  \cite{A71,GH81} we have:

\begin{prop}\label{prop.fixedspin}
Given  $\Sigma$ a compact oriented surface and $f : \Sigma \to \Sigma$ an orientation reversing involution,  there exists a spin structure preserved by $f$.
\end{prop}

There is also a natural action of $f$ on the theta characteristics. Since $f$ is anti-holomorphic there is a natural isomorphism $f^*(\overline{K}) \simeq K$. It follows that if $L$ is a theta characteristic, then so is $f^*(\overline{L})$.

\begin{prop}\label{prop.actionspin}
The action of $f$ on the set of spin structures of the Riemann surface $\Sigma$ interpreted as theta characteristics agrees with the action on $H^1(U\Sigma , \mathbb{Z}_2)$ induced by the differential   $f_* : U\Sigma \to U\Sigma$.
\end{prop}
\begin{proof}
Let $L$ be a holomorphic line bundle such that $L^2 \simeq K$. Then $L$ inherits a hermitian metric $h$ from the hermitian metric on $K$. Let $Q \to U\Sigma$ be the corresponding double cover of $U\Sigma$ and $\xi \in H^1(U\Sigma , \mathbb{Z}_2)$ the class defined by $Q$. The hermitian metric $h$ on $L$ determines a corresponding metric on $f^*(\overline{L})$ and it is clear that the corresponding double cover of $U\Sigma$ is isomorphic to $f^*(Q)$ which corresponds to $f^*(\xi) \in H^1(U\Sigma,\mathbb{Z}_2)$.
\end{proof}

From the above propositions  we conclude the existence of theta characteristics $K^{1/2}$ such that $f^*( \overline{K}^{1/2} ) \simeq K^{1/2}$.
 

\subsection{Action on the fundamental group}\label{sec:ac.fund} Since an anti-holomorphic involution of the Riemann surface  $f:\Sigma\rightarrow \Sigma$ is a homeomorphism, it induces an isomorphism
\begin{equation*}
 f_* : \pi_1(\Sigma, x_0) \to \pi_1(\Sigma, f(x_0))
\end{equation*}
where $x_0 \in \Sigma$. Given a path $\gamma$ joining $x_0$ to $f(x_0)$,   conjugating by $\gamma$ determines an isomorphism
\begin{equation*}
\phi_\gamma : \pi_1(\Sigma, f(x_0)) \to \pi_1(\Sigma, (x_0))
\end{equation*} which sends a loop $u$ based at $f(x_0)$ to the loop $\gamma . u . \gamma^{-1}$ based at $x_{0}$, where we use $.$ to denote the operation of joining paths. Hence the composition
\begin{equation*}
\hat{f} = \phi_\gamma \circ f_* : \pi_1(\Sigma , x_0) \to \pi_1(\Sigma , x_0)
\end{equation*} is an automorphism of $\pi_1(\Sigma , x_0)$. From the above, different choices of $\gamma$ change $\hat{f}$ by composition with an inner automorphism. Observe that $f(\gamma)$ is a path from $f(x_0)$ to $x_0$ so the composition $h = \gamma . f(\gamma)$ is a loop based at $x_0$. A straightforward computation shows that $\hat{f}^2$ is the inner automorphism $u \mapsto huh^{-1}$.
 
If the map $f$ has fixed points,   we can choose $x_0$ to be a fixed point and $\gamma$ to be the constant path. Then $h$ is the trivial loop and $\hat{f}$ is an involutive automorphism of $\pi_1(\Sigma)$. On the contrary, if $f$ is fixed point free then $\Sigma$ is a double cover $\pi : \Sigma \to \Sigma'$ of a non-orientable surface $\Sigma'$ and thus we get an exact sequence
\begin{equation*}
1 \to \pi_1(\Sigma) \to \pi_1(\Sigma') \to \mathbb{Z}_2 \to 1.
\end{equation*}
The image $\gamma' = \pi(\gamma)$ of $\gamma$ is a class in $\pi_1(\Sigma')$ lifting the generator of $\mathbb{Z}_2$. The automorphism $\hat{f}$ is conjugation by $\gamma'$ in the sense that $\pi_* \hat{f}(v) = \gamma' \pi_*(v) (\gamma')^{-1}.$ In particular, $\hat{f}^2$ is conjugation by $(\gamma')^2$.



\subsection{Action on principal bundles} \label{sec:indprin}

Let $P \to \Sigma$ be a principal $G_c$-bundle, for $G_{c}$ a complex connected Lie group, and $f^*(P)$ be the pullback principal bundle. 
Since $G_{c}$ is connected, the principal bundle  $P$ can be trivialized over the $1$-skeleton of $\Sigma$. 
 The obstruction to extending the trivialization to the $2$-skeleton is a cohomology class in $H^2( \Sigma , \pi_1(G_c)) \simeq \pi_{1}(G_{c})$. 
Moreover from \cite[Proposition 5.1]{ram} one has the following:

\begin{prop}\textcolor{black}{Isomorphism classes of } $G_c$-bundles are in bijection with $H^2( \Sigma , \pi_1(G_c))$.
\end{prop}

From the above proposition, the action of $f^*$ can be seen to be the pullback in cohomology $$f^* : H^2(\Sigma , \pi_1(G_c)) \to H^2(\Sigma , \pi_1(G_c)).$$ Moreover, since $H^2(\Sigma , \pi_1(G_c)) \simeq H^2(\Sigma,\mathbb{Z}) \otimes \pi_1(G_c)$, the induced action $f^*$ is multiplication by $-1$. Therefore a principal $G_{c}$-bundle $P \to \Sigma $ is fixed by $f$ if and only if its topological class $x \in \pi(G_{c})$ is a 2-torsion element.

 




%
%


\section{Action on representations and the fixed point set}\label{sec:action-moduli}

In this section we shall denote by $G_c$  a complex connected Lie group with Lie algebra $\mathfrak{g}$, and assume that $\mathfrak{g}$ admits an invariant symmetric non-degenerate bilinear form $B$.


\subsection{Moduli space of representations}\label{msr}

Let $\Sigma$ be a compact orientable surface of genus $g > 1$ and let $\pi = \pi_1(\Sigma , x_0)$ be the fundamental group. We denote by ${\rm Hom}(\pi , G_c)$ the set of homomorphisms $\rho : \pi \to G_c$ given the compact-open topology. Given a representation $\rho : \pi \to G_c$, composition with the adjoint representation of $G_c$ defines a representation of $\pi$ on $\mathfrak{g}$. We denote by $\mathfrak{g}_\rho$ the Lie algebra $\mathfrak{g}$ equipped with this representation.

The space ${\rm Hom}(\pi , G_c)$ has a natural action of $G_c$ by conjugation. To obtain a good quotient one restricts to the subspace ${\rm Hom}^+(\pi , G_c)$ of reductive representations. Recall that a representation $\rho : \pi \to G_c$ is called reductive if $\mathfrak{g}_\rho$ splits into a direct sum of irreducible representations.
Restricted to reductive representations the action of $G_c$ on ${\rm Hom}^+(\pi , G_c)$ is proper \cite{goldman}, and thus the quotient is a Hausdorff space. 

Let ${\rm Rep}^+( \pi , G_c)$ be the quotient space ${\rm Hom}^+(\pi , G_c)/G_c$, the moduli space of reductive $G_c$-representations of $\pi_1(\Sigma)$. In general ${\rm Rep}^+(\pi , G_c)$ has singularities but there is a dense open subset of smooth points over which ${\rm Rep}^+(\pi , G_c)$ naturally has the structure of a complex manifold. For a representation $\rho$, let $G_\rho \subseteq G_c$ be the stabilizer of $\rho$ and let $Z(G_c)$ be the centre of $G_c$. If $\rho$ is a reductive representation such that ${\rm dim}(G_\rho) = {\rm dim}(Z(G_c))$, then $\rho$ is a smooth point of ${\rm Hom}^+(\pi , G_c)$ and the dimension of this space is $(2g-1){\rm dim}(G_c) + {\rm dim}(Z(G_c))$ \cite{goldman}. This condition also ensures that the stabilizer $G_\rho / Z(G_c)$ of $\rho$ in $G_c/Z(G_c)$ is discrete, hence finite since the action is proper. It follows from general theory (see \cite{mm}) that around such points the quotient space ${\rm Rep}^+(\pi , G_c)$ has the structure of an orbifold. Moreover, if $\rho$ is a reductive representation with $G_\rho = Z(G_c)$, then the corresponding point in ${\rm Rep}^+(\pi , G_c)$ is smooth of dimension $(2g-2){\rm dim}(G_c) + 2{\rm dim}(Z(G_c))$. 

\begin{definition}
We say that a reductive representation $\rho$ is simple if $G_\rho = Z(G_c)$. In particular if the representation of $\pi$ on $\mathfrak{g}$ induced by $\rho$ is irreducible, then $\rho$ is simple.
\end{definition}

If $\rho \in {\rm Rep}^+(\pi , G_c)$ is a simple point, then the tangent space at $\rho$ is naturally given in terms of group cohomology by $H^1( \pi , \mathfrak{g}_\rho)$. The pairing $B : \mathfrak{g} \otimes \mathfrak{g} \to \mathbb{C}$ determines a symplectic pairing
\begin{equation*}
H^1( \pi , \mathfrak{g}_\rho ) \otimes H^1( \pi , \mathfrak{g}_\rho ) \to H^2( \pi , \mathbb{C} ) \simeq \mathbb{C}.
\end{equation*}
This defines a closed complex symplectic form $\Omega$ on the smooth part of ${\rm Rep}^+( \pi , G_c)$ \cite{goldman}.


\subsection{Induced action on space of representations}\label{msr2}

From Section \ref{sec:ac.fund} we have seen that given an orientation reversing involution $f : \Sigma \to \Sigma$, there is an induced map $\hat{f} : \pi \to \pi$ defined up to an inner automorphism. Accordingly there is an induced action on ${\rm Hom}(\pi , G_c)$ sending a representation $\rho$ to $\rho \circ \hat{f}$. This action preserves the subspace of reductive representations and descends through the quotient to an action $f : {\rm Rep}^+(\pi , G_c) \to {\rm Rep}^+( \pi , G_c)$. We may identify ${\rm Rep}^+(\pi , G_c)$ with the moduli space of gauge equivalence classes of flat $G_c$-connections on $\Sigma$ with reductive holonomy. Then  $f : {\rm Rep}^+(\pi , G_c) \to {\rm Rep}^+( \pi , G_c)$ corresponds to the  pullback of connections by $f : \Sigma \to \Sigma$.

\begin{prop}\label{prop:antisymplectic}
The induced map $f : {\rm Rep}^+(\pi , G_c) \to {\rm Rep}^+( \pi , G_c)$ is an involution which preserves the subspace of simple points. On the simple points $f$ is holomorphic and anti-symplectic, that is $f^* \Omega = -\Omega$.
\end{prop}

\begin{proof}
Recall that the automorphism $\hat{f} : \pi \to \pi$ is such that $\hat{f}^2$ is an inner automorphism. It follows that the induced map $f : {\rm Rep}^+(\pi , G_c) \to {\rm Rep}^+(\pi , G_c)$ is an involution which  preserves the space of simple points. It is clear that $f$ acts smoothly on the space of simple points, and its differential is  the induced pullback in group cohomology $f^* : H^1( \pi , \mathfrak{g}_\rho ) \to H^1( \pi , \mathfrak{g}_{ \rho \circ \hat{f}})$. Since this is a complex linear map, then  $f : {\rm Rep}^+(\pi , G_c) \to {\rm Rep}^+( \pi , G_c)$ is holomorphic on the simple points. Moreover $f$ is orientation reversing from which the anti-symplectic condition $f^* \Omega = -\Omega$ follows.
\end{proof}


\begin{lem}\label{lem:fixedpoints}
Let $X$ be a complex manifold of dimension $n$ and $f : X \to X$ an anti-holomorphic involution. If the fixed point set of $f$ is non-empty, then it is a smooth analytic submanifold of real dimension $n$.
\end{lem}
\begin{proof}
Let $p$ be a fixed point of $f$ and let $g$ be a Riemannian metric on $X$. Replacing $g$ by $g + f^*(g)$ we may assume that $g$ is $f$-invariant. We can also assume that $g$ is analytic near $p$. Let $e : T_pX \to X$ be the exponential mapping at $p$, and  $C = f_*(p) : T_pX \to T_pX$ the differential of $f$ at $p$. Then $C$ is an anti-linear involution on $T_pX$, and since $g$ is $f$-invariant we have $e( C(v)) = f(e(v))$ for any $v \in T_pX$. Thus a neighborhood of the origin in $T_pX$ defines local analytic coordinates near $p$ such that $f$ corresponds to the anti-linear involution $C$. The fixed point set of $f$ near $p$ is thus a smooth analytic submanifold of real dimension $n$.
\end{proof}

Let $\mathcal{L}_{G_c}$ denote the fixed point set of the involution $f : {\rm Rep}^+(\pi , G_c) \to {\rm Rep}^+(\pi , G_c)$. From Proposition \ref{prop:antisymplectic} and Lemma \ref{lem:fixedpoints} we obtain:
\begin{prop}\label{prop:complexlag}
If non-empty, the set of simple points of $\mathcal{L}_{G_c}$ is a smooth complex Lagrangian submanifold of the simple points of ${\rm Rep}^+(\pi,G_c)$.
\end{prop}

Through the theory of spectral curves in Section \ref{sec:invo-spec}, we show that $\mathcal{L}_{G_c}$ does indeed contain smooth points for $G_c = GL(n,\mathbb{C})$, and for $G_{c}$ a classical complex semi-simple Lie group.


\section{Moduli space of $G_{c}$-Higgs bundles}\label{sec:inro-Higgs}


A Higgs bundle on a Riemann surface $\Sigma$ is a pair $(E,\Phi)$, where $E$ is a holomorphic vector bundle on $\Sigma$ and $\Phi$ is a holomorphic section of ${\rm End}(E) \otimes K$. More generally for a complex Lie group $G_{c}$ we define a $G_{c}$-Higgs bundle to be a pair $(P,\Phi)$ where $P$ is a holomorphic principal $G_c$-bundle with adjoint bundle ${\rm ad}(P)$ and $\Phi$ is a holomorphic section of ${\rm ad}(P) \otimes K$. In the following subsections we shall study the geometry of the moduli space of $G_{c}$-Higgs bundles, which is closely related to the moduli space of surface group representations into $G_{c}$.

\subsection{Higgs bundles and the Hitchin equations}\label{sec:higgs}

For simplicity assume $G_c$ is connected and semisimple.
We shall denote by $\mathfrak{g}_c$ the Lie algebra of $G_c$, and $\mathfrak{g}$ the Lie algebra of $G$.
Given  a choice of principal $G_c$-bundle $P$ with reduction of structure to the compact real form $G$ of $G_{c}$,  the Killing form $k(x,y)$ on $\mathfrak{g}_c$ naturally defines a bilinear form on the adjoint bundle ${\rm ad}(P)$ which we will also denote by $k$. A reduction of structure of $P$ to $G$ amounts to equipping the adjoint bundle ${\rm ad}(P)$ with an anti-linear involution $\rho : {\rm ad}(P) \to {\rm ad}(P)$ such that taking the hermitian adjoint $x^*$ of a section $x$ of ${\rm ad}(P)$ is given by $x^* = -\rho(x)$. The associated hermitian form $h$ is given by $h(x,y) = k(x^*,y) = -k(\rho(x),y).$ 
Note that  $\overline{ k(x,y) } = k(x^*,y^*)$.

Consider pairs $(\overline{\partial}_A , \Phi)$, where $\overline{\partial}_A$ denotes a $\overline{\partial}$-connection on $P$ defining a holomorphic structure, and $\Phi$ is a section of $\Omega^{1,0}(\Sigma , {\rm ad}(P))$. Note that if $\overline{\partial}_A \Phi = 0$ then $(\overline{\partial}_A , \Phi)$ defines a $G_c$-Higgs bundle on $P$. From \cite{N1}, the Hitchin equations for a pair $(\overline{\partial}_A , \Phi)$ are
\begin{equation}
\begin{aligned}
\overline{\partial}_A \Phi &= 0, \\
F_A + [\Phi , \Phi^*] &= 0,
\end{aligned}\label{Heq}
\end{equation}
where $F_A$ is the curvatuve of the unitary connection $\nabla_A = \partial_A + \overline{\partial}_A$ associated to $\overline{\partial}_A$.
Two solutions to the Hitchin equations on $P$ are considered equivalent if they are related by a $G$-valued gauge transform. We let $\mathcal{M}_{G_c}(P)$ denote the moduli space of gauge equivalence classes of solutions to the Hitchin equations on $P$, and $\mathcal{M}_{G_c}$ the union of the $\mathcal{M}_{G_c}(P)$ as $P$ ranges over the set of isomorphism classes of principal $G_c$-bundles.

A solution $(\overline{\partial}_A , \Phi)$ to the Hitchin equations determines a flat $G_c$-connection $\nabla = \nabla_A + \Phi + \Phi^*$. Using results of Donaldson \cite{donald} and Corlette \cite{cor},   the assignment $(\overline{\partial}_A , \Phi) \mapsto \nabla$ gives an isomorphism between the Higgs bundle moduli space $\mathcal{M}_{G_c}$ of equivalence classes of solutions to the Hitchin equations and the moduli space ${\rm Rep}^+( \pi_1(\Sigma) , G_c)$ of reductive representations of $\pi_1(\Sigma)$ in $G_c$. In particular every flat $G_c$-connection $\nabla$ with reductive holonomy representation admits a decomposition $\nabla = \nabla_A + \Phi + \Phi^*$ associated to a solution $(\overline{\partial}_A , \Phi )$ of the Hitchin equations.

The slope of a holomorphic vector bundle $E \to \Sigma$ is defined as $\mu(E) := {\rm deg}(E) / {\rm rank}(E)$. We say that a Higgs bundle $(E,\Phi)$ is semi-stable if for each proper, non-zero subbundle $F \subset E$ which is $\Phi$-invariant we have $\mu(F) \le \mu(E)$. If this inequality is always strict then $(E,\Phi)$ is said to be stable. Finally, the Higgs bundle $(E,\Phi)$ is poly-stable if it is a sum of stable Higgs bundles of the same slope. It is possible to adapt these definitions to the case of principal $G_c$-Higgs bundles for a complex semisimple Lie group $G_c$ (e.g., see \cite[Section 3]{biswas}).

One may define a moduli space of semi-stable $G_c$-Higgs bundles $\CM_{G_c}^{{\rm Higgs}}$, and by a fundamental result of Hitchin \cite{N1} and Simpson \cite{simpson88} there is an isomorphism between $\CM_{G_c}^{{\rm Higgs}}$ and the moduli space $\mathcal{M}_{G_c}$ of solutions to the Hitchin equations, when $G_c$ is semi-simple. The key result used to establish this is that a $G_c$-Higgs bundle $(\overline{\partial}_A,\Phi)$ is gauge equivalent to a solution of the Hitchin equations if and only if it is poly-stable. A similar isomorphism exists when the group is $G_c = GL(n,\mathbb{C})$, except that in this case $\CM_{GL(n,\mathbb{C})}$ corresponds to the subspace $\CM_{GL(n,\mathbb{C})}^{{\rm Higgs},0} \subset \CM_{G_c}^{{\rm Higgs}}$ of Higgs bundles of degree $0$.


\subsection{Geometry of the Higgs bundle moduli space}

The moduli space $\mathcal{M}_{G_c}$ of $G_{c}$-Higgs bundles, for $G_{c}$ a complex semisimple Lie group,  is a  hyperk\"ahler manifold with singularities obtained by taking a hyperk\"ahler quotient of the infinite dimensional space of pairs of complex structures and Higgs fields (\cite{N1}, \cite{simpson88}).

Recall that a pair $(\overline{\partial}_A , \Phi)$ consists of a holomorphic structure $\overline{\partial}_A$ on the principal bundle $P$, and a section $\Phi \in \Omega^{1,0}(\Sigma , {\rm ad}(P))$. The space of all pairs $( \overline{\partial}_A , \Phi)$  on $P$ is an infinite dimensional manifold and an affine space modelled on $\Omega^{0,1}(\Sigma , {\rm ad}(P) ) \oplus \Omega^{1,0}(\Sigma , {\rm ad}(P))$. We write $(\Psi_1 , \Phi_1) , (\Psi_2 , \Phi_2) , \dots $ for tangent vectors to this space. Furthermore, given a pair $( \overline{\partial}_A , \Phi)$ we write $\nabla_A$ for the unitary connection corresponding to $\overline{\partial}_A$, and $F_A$ its curvature. Explicitly $\nabla_A = \overline{\partial}_A + \partial_A$ where $\partial_A = \rho \circ \overline{\partial}_A \circ \rho$.
The metric on this infinite dimensional space is given by
\begin{equation}
g( (\Psi_1,\Phi_1) , (\Psi_1 , \Phi_1) ) = 2i \int_{\Sigma} k( \Psi_1^* , \Psi_1) - k( \Phi_1^* , \Phi_1).\label{eq.met}
\end{equation}
There are compatible complex structures $I,J,K$ satisfying the quaternionic relations given by
\begin{eqnarray}
 I (\Psi_1,\Phi_1) &=& (i\Psi_1,i\Phi_1); \nonumber\\
J (\Psi_1,\Phi_1) &= &(i\Phi_1^*,-i\Psi_1^*);\label{quat.eq} \\
K (\Psi_1,\Phi_1) &=& (-\Phi_1^*,\Psi_1^*).\nonumber
 \end{eqnarray}
We shall denote by $\omega_I,\omega_J,\omega_K$ the corresponding K\"ahler forms defined by 
\begin{eqnarray*}\omega_I(X,Y) := g(IX,Y)~,~\omega_J(X,Y) := g(JX,Y)~,~ \omega_K(X,Y) := g(KX,Y). \end{eqnarray*}
The induced complex symplectic forms $\Omega_I,\Omega_J,\Omega_K$  are given by 
\begin{eqnarray*}\Omega_I = \omega_J + i \omega_K~,~ \Omega_J = \omega_K + i\omega_I~, ~\Omega_K = \omega_I + i\omega_J.\end{eqnarray*}
\begin{ex}
For example $\Omega_J$ is given by:
\begin{equation*}
\Omega_J( (\Psi_1,\Phi_1) , (\Psi_2,\Phi_2) ) = -i\int_\Sigma k( \Psi_1 - \Psi_1^* + \Phi_1 + \Phi_1^* , \Psi_2 - \Psi_2^* + \Phi_2 + \Phi_2^*).
\end{equation*}
After taking the  hyperk\"ahler quotient this agrees up to the factor of $-i$ with the symplectic form $\Omega$ defined in Section \ref{msr}, that is $\Omega_J = -i \Omega$ on $\mathcal{M}_{G_c}$.
\end{ex}

To get the moduli space $\mathcal{M}_{G_c}$, as seen in \cite{N1} one interprets the Hitchin equations (\ref{Heq}) as the moment map equations for the action of the unitary gauge group and takes the hyperk\"ahler quotient.
Given $( \overline{\partial}_A , \Phi)$ a solution of the Hitchin equations (\ref{Heq}), the tangent space $T_{( \overline{\partial}_A , \Phi)} \mathcal{M}_{G_c}$ of the moduli space at a smooth point $( \overline{\partial}_A , \Phi)$ can be described as equivalence classes of deformations $(\Psi_1 , \Phi_1) $ in $\Omega^{0,1}(\Sigma , {\rm ad}(P) ) \oplus \Omega^{1,0}(\Sigma , {\rm ad}(P))$ satisfying:
\begin{equation*}
\begin{aligned}
\overline{\partial}_A \Phi_1 + [ \Psi_1 , \Phi] &= 0, \\
\partial_A \Psi_1 - \overline{\partial}_A \Psi_1^* + [\Phi , \Phi_1^*] + [\Phi_1 , \Phi^*] &= 0.
\end{aligned}
\end{equation*}
Two deformations are equivalent if they are related by an infinitesimal unitary gauge transformation. Thus $(\Psi_2 , \Phi_2)$ is equivalent to $(\Psi_1 , \Phi_1)$ if there is a skew-adjoint section $\psi \in \Omega^0(\Sigma , {\rm ad}(P))$ of the adjoint bundle such that
\begin{equation*}
\begin{aligned}
\Psi_2 &= \Psi_1 + \overline{\partial}_A \psi, \\
\Phi_2 &= \Phi_1 + [\Phi , \psi].
\end{aligned}
\end{equation*}

\section{Higgs bundles and the $(A,B,A)$-brane $\CL_{G_{c}}$}\label{sec:aba}

Given $f: \Sigma \to \Sigma$ an anti-holomorphic involution on $\Sigma$, we saw in Section \ref{sec:action-moduli} that $f$ induces an involution on ${\rm Rep}^+(\pi_1(\Sigma) , G_c)$ with fixed point set $\mathcal{L}_{G_c}$. By identifying ${\rm Rep}^+(\pi_1(\Sigma) , G_c)$ with $\mathcal{M}_{G_c}$ we obtain an involution on the moduli space of $G_c$-Higgs bundles. From this perspective we can interpret the fixed point set $\mathcal{L}_{G_c}$ in terms of $B$-branes and $A$-branes following \cite{Kap}.

Let $P$ be a principal $G_c$-bundle over $\Sigma$ and fix a reduction of structure of $P$ to $G$. For $x \in \mathcal{M}_{G_c}$ a point in the moduli space represented by a $G_c$-Higgs bundle pair $(\overline{\partial}_A , \Phi)$ on $P$, applying a gauge transform we may assume that $(\overline{\partial}_A , \Phi)$ satisfies the Hitchin equations. In particular the flat connection $\nabla$ corresponding to $(\overline{\partial}_A , \Phi)$ is given by $\nabla = \partial_A + \overline{\partial}_A + \Phi + \Phi^*$. The involution $f$ acts on this flat connection by pullback, so we obtain $f^* \nabla = f^*(\partial_A) + f^*(\overline{\partial}_A) + f^*(\Phi) + f^*(\Phi^*)$. Therefore $(f^*(\partial_A) , f^*(\Phi^*))$ is a $G_c$-Higgs bundle on $f^*(P)$ which satisfies the Hitchin equations. The flat connection associated to the pair $(f^*(\partial_A) , f^*(\Phi^*))$ is $f^*\nabla$, so we have found a Higgs bundle pair $(f^*(\partial_A) , f^*(\Phi^*))$ representing $f(x) \in \mathcal{M}_{G_c}$.

\begin{prop}\label{ind.inv}\label{prop:LJ1}
The induced involution $f : \mathcal{M}_{G_c} \to \mathcal{M}_{G_c}$ on the moduli space of solutions $(\overline{\partial}_A , \Phi )$ to the $G_{c}$-Hitchin equations (\ref{Heq}) is holomorphic with respect to the complex structure $J$ and anti-holomorphic with respect to the complex structures $I$ and $K$. Moreover $f $ is an isometry with respect to the hyperk\"ahler metric.
\end{prop}
\begin{proof}
As seen above, the action of $f$ on $\mathcal{M}_{G_c}$ is given by sending a solution $(\overline{\partial}_A , \Phi)$ of the Hitchin equations to the pair $(f^*(\partial_A) , f^*(\Phi^*))$. The differential of $f$ at $(\overline{\partial}_A , \Phi )$ sends a deformation $(\Psi_1 , \Phi_1)$ of $(\overline{\partial}_A , \Phi )$ to a corresponding deformation of $(f^*(\partial_A) , f^*(\Phi^*))$ given by $( -f^*( \Psi_1^*) , f^*(\Phi_1^*))$. From this and (\ref{quat.eq}) we see that $f$ is anti-holomorphic with respect to $I,K$ and holomorphic with respect to $J$. Similarly, it follows from  (\ref{eq.met}) that $f$ is an isometry.
\end{proof}

\begin{prop}\label{prop:LJ}\label{prop:LJ2}
The fixed point set $\mathcal{L}_{G_{c}}$ of $f$ on the moduli space of solutions $(\overline{\partial}_A , \Phi )$ to the $G_{c}$-Hitchin equations (\ref{Heq}) meets the smooth points in a complex Lagrangian submanifold with respect to $J,\Omega_J$.
\end{prop}
\begin{proof}Since $f$ is anti-holomorphic in $I$, we know immediately that its fixed point set $\mathcal{L}_{G_c}$ is a mid-dimensional submanifold of $\mathcal{M}_{G_c}$. Clearly $\mathcal{L}_{G_c}$ is a complex submanifold with respect to $J$, and the symplectic forms $\omega_I,\omega_K$ must vanish on $\mathcal{L}_{G_c}$. In particular this means the complex symplectic form $\Omega_J = \omega_K + i\omega_I$ vanishes on $\mathcal{L}_{G_c}$. Thus $\mathcal{L}_{G_c}$ is a complex Lagrangian submanifold with respect to $J$. \end{proof}

Following \cite{Kap}, we say that $\mathcal{L}_{G_c}$ is an $(A,B,A)$-brane with respect to the complex structures $I,J$ and $K$, and thus have the following:
 
 \begin{theorem}\label{teo1}
 For each choice of anti-holomorphic involution $f$ on a compact Riemann surface, there is a natural $(A,B,A)$-brane $\CL_{G_{c}}$ defined in the moduli space of $G_{c}$-Higgs bundles. 
 \end{theorem}
 
In subsequent sections we shall describe how the brane $\CL_{G_{c}}$ lies with respect to the Hitchin fibration for the moduli space $\CM_{G_{c}}$ (Section \ref{sec:hitFib}), consider the spectral data for $G_{c}$-Higgs bundle to show that $\CL_{G_{c}}$ is non-empty (Section \ref{sec:invo-spec}), and study connectivity of $\CL_{G_{c}}$ (Section \ref{sec:con}).

\section{The Hitchin fibration and $\CL_{G_{c}}$}\label{sec:hitFib}


Let $R^* = \mathbb{C}[\mathfrak{g}_c^*]^{G_c}$ be the graded $\mathbb{C}$-algebra of invariant polynomials on $\mathfrak{g}_c$ and ${\rm Sym}^*$ the graded $\mathbb{C}$-algebra with ${\rm Sym}^j = H^0(\Sigma , K^j)$. Let $\mathcal{A}_{G_c} = Hom( R^* , {\rm Sym}^* )$ be the space of graded $\mathbb{C}$-algebra homomorphisms from $R^*$ to ${\rm Sym}^*$. Choosing a homogeneous basis of generators $p_1,p_2, \dots p_l$ for $R^*$, we have $R^* \simeq \mathbb{C}[p_1, \dots , p_l]$ and hence a non-canonical isomorphism
\begin{equation*}
\mathcal{A}_{G_c} = \bigoplus_{i=1}^l H^0( \Sigma , K^{d_i}),
\end{equation*}
where $d_i$ is the degree of $p_i$. From \cite{Hit92}, we consider the Hitchin fibration $h : \mathcal{M}_{G_c} \to \mathcal{A}_{G_c}$, which assigns to a Higgs bundle $(P,\Phi)$ the homomorphism $h(P,\Phi) : R^* \to Sym^*$ sending an invariant polynomial $p$ to its evaluation $p(\Phi)$ on $\Phi$.

The anti-holomorphic involution $f$ on the compact  Riemann surface  $\Sigma$ induces  a natural anti-holomorphic involution on the spaces $H^0(\Sigma , K^j)$ given by sending a holomorphic differential $q$ to $\overline{ f^*(q) }$. The induced map $f : {\rm Sym}^* \to {\rm Sym}^*$ is an anti-linear graded ring involution. We may also define an anti-linear graded ring involution $f : R^* \to R^*$ as follows. Given  $p \in R^*$ and $x \in \mathfrak{g}_c$, set $(fp)(x) = -\overline{p( \rho(x) )}$, for $\rho$ the  compact anti-involution defined as in Section \ref{sec:inro-Higgs}. Combining these actions on $R^*$ and ${\rm Sym}^*$ we obtain an anti-linear involution $f : \mathcal{A}_{G_c} \to \mathcal{A}_{G_c}$.

Let $(P,\Phi)$ be a Higgs bundle pair which satisfies the Hitchin equations. Then $f(P,\Phi)$ has Higgs field $f^*(\Phi^*) = -f^*(\rho \Phi)$. Given $p$  any invariant polynomial, we find that
\begin{equation*}
\begin{aligned}
h( f(P,\Phi) ) (p) &= p( -f^*( \rho \Phi ) ) \\
&= - f^*( p \rho \Phi ) \\
& = f^*( \overline{ (fp)(\Phi) } ).
\end{aligned}
\end{equation*}
Hence, the Hitchin map commutes with the involutions:
\begin{equation*}\xymatrix{
\mathcal{M}_{G_c} \ar[r]^f \ar[d]^h & \mathcal{M}_{G_c} \ar[d]^h \\
\mathcal{A}_{G_c} \ar[r]^f & \mathcal{A}_{G_c}
}
\end{equation*}

\begin{remark}\label{inv:gl} In the case that $G_c = GL(n,\mathbb{C})$ a generating basis of invariant polynomials is given by the traces of powers $p_j(x) = {\rm tr}(x^j)$, and $$(fp_j)(x) = - \overline{ {\rm tr}( \rho(x)^j ) } = {\rm tr}( (x^j)^{t} ) = p_j(x),$$ where we have used $\rho x = -x^* = -\overline{x}^{t}$. Hence, for $GL(n,\mathbb{C})$ the above is a basis of generators for the invariant polynomials which are fixed by the involution $f : R^* \to R^*$. \end{remark}

For $L \subset \mathcal{A}_{G_c}$ the fixed point set of $f$ on $\mathcal{A}_{G_c}$,  it follows that the Hitchin fibration restricts to a map $h|_{\mathcal{L}_{G_c}} : \mathcal{L}_{G_c} \to L$. For simplicity, we shall drop the subscript  of $\mathcal{L}_{G_{c}}$, and refer to the $(A,B,A)$-brane $\mathcal{L}\subset \CM_{G_{c}}$. 

\begin{prop}\label{prop:sub}
Away from the singular fibres of the Hitchin map, the restriction $h|_{\mathcal{L}} : \mathcal{L} \to L$ is a submersion.
\end{prop}
\begin{proof}
Recall from Proposition \ref{prop:LJ} that $\mathcal{L} $ is a holomorphic submanifold of $\mathcal{M}_{G_c}$ with respect to $J$, so it is naturally a complex K\"ahler manifold with complex structure $J|_{\mathcal{L}}$ and K\"ahler form $\omega_J|_{\mathcal{L}}$. Moreover, from \cite{Hit92} the (non-singular) fibres of the Hitchin map are Lagrangian with respect to the complex symplectic form $\Omega_I = \omega_J + i\omega_K$. Thus, in particular, $\omega_J$ vanishes on the fibres of the Hitchin map. 

For any smooth point $x \in \mathcal{L}$ the kernel $K_x$ of $(h|_{\mathcal{L}})_*$ is an isotropic subspace of $T_x \mathcal{L}$ with respect to $\omega_J|_{\mathcal{L}}$, and thus ${\rm dim}(K_x) \le \frac{1}{2} {\rm dim}(\mathcal{L})$. But on the other hand the image $I_x$ of $(h|_{\mathcal{L}})_*$ lies in $L$, so we also have ${\rm dim}(I_x) \le {\rm dim}(L) = \frac{1}{2} {\rm dim}(\mathcal{L})$. So for any $x \in \mathcal{L}$ the dimensions satisfy
\begin{equation*}
{\rm dim}(\mathcal{L}) = {\rm dim}(K_x) + {\rm dim}(I_x) \le \frac{1}{2} {\rm dim}(\mathcal{L}) + \frac{1}{2} {\rm dim}(\mathcal{L}) = {\rm dim}(\mathcal{L}).
\end{equation*}
Hence, since there must be equalities throughout, the restriction $h|_{\mathcal{L}}$ is a submersion.
\end{proof}

It follows from the  above proof that away from the singular fibres of the Hitchin fibration, the fibres of $h|_{\mathcal{L}} : \mathcal{L} \to L$ are Lagrangian submanifolds of $\mathcal{L}$, where the symplectic structure on $\mathcal{L}$ is given by $\omega_J |_{\mathcal{L}}$.

\begin{thm}\label{thm:intsys}
If $\mathcal{L}_{G_c}$ contains smooth points then the restriction of the Hitchin fibration $$h|_{\mathcal{L}_{G_c}} : \mathcal{L}_{G_c} \to L$$ to $\mathcal{L}_{G_c}$ is a Lagrangian fibration with singularities. The generic fibre is smooth and consists of a finite number of tori.
\end{thm}
\begin{proof}
Assume that $\mathcal{L}$ contains smooth points. Then $\mathcal{L}$ meets the smooth points of the Higgs bundle moduli space $\mathcal{M}_{G_c}$ in a smooth real analytic submanifold. At a smooth fixed point $p$ the tangent space to $\mathcal{L}$ defines a real structure on $T_p \mathcal{M}_{G_c}$. On the other hand the smooth points where the Hitchin map is not a submersion form a complex analytic subvariety $V$. By considering power series one sees that the smooth points of $\mathcal{L}$ can not lie entirely in $V$, so that the Hitchin map is a submersion at a generic smooth point of $\mathcal{L}$. This also shows that the image $h( U )$ of any non-empty open subset $U$ of $\CL$ does not lie entirely in the space of critical values of $h$. Hence the generic fibre of the restricted Hitchin map $h|_{\mathcal{L}} : \mathcal{L} \to L$ is smooth, giving a Lagrangian fibration with singularities. On the generic smooth fibres of $h|_{\mathcal{L}}$ we obtain ${\rm dim}(\mathcal{L})/2$ commuting vector fields, which are complete since $h$ is a proper map (eg: \cite{N1} for the rank $2$ case). It follows that the generic fibres of $h|_{\mathcal{L}}$ consist of finitely many copies of a torus.
\end{proof}

\begin{rem}
Let $\Delta \subset L$ denote the points of $L$ over which the Hitchin map is not a submersion. Each fibre of $\mathcal{L}$ over $L \setminus \Delta$ consists of a finite number of tori. In general $L \setminus \Delta$ may be disconnected and the number of tori in a given fibre varies as one moves between the components of $L \setminus \Delta$.
\end{rem}

The above analysis establishes that the fixed point set $\mathcal{L}$ has the structure of a real integrable system with singularities.


\section{Spectral data for $\CL_{G_{c}}$} \label{sec:invo-spec}

As seen in  Proposition \ref{prop:LJ1}, the map $f$ induces an involution on the moduli space $\mathcal{M}_{G_{c}}$ of polystable $G_{c}$-Higgs bundles
\begin{eqnarray}
f: \mathcal{M}_{G_{c}}&\rightarrow &\mathcal{M}_{G_{c}}\nonumber \\
(\overline{\partial}_A, \Phi) &\mapsto & (f^*(\partial_A), f^*(\Phi^*)   ). \nonumber
\end{eqnarray}
In this section we shall describe how $f$ acts in terms of the spectral data associated to $G_{c}$-Higgs bundles, for $G_{c}$ a classical complex Lie group and see that the fixed point set $\mathcal{L}$ is non-empty (see e.g., \cite{Hit07} for the construction of the spectral data).

\begin{rem}
In the case of $G_c = GL(n,\mathbb{C})$ the action $(\overline{\partial}_A, \Phi) \mapsto (f^*(\partial_A), f^*(\Phi^*) )$ makes sense for Higgs bundles of any degree, so that the action of $f$ can be extended to the full moduli space of $GL(n,\mathbb{C})$-Higgs bundles of arbitrary degree. However, we shall see that there can only be fixed points in degree $0$. This extended action can be interpreted as the induced action of $f$ on representations of a central extension of $\pi_1(\Sigma)$, following \cite{atbott}.
\end{rem}


\subsection{Spectral data for $GL(n,\mathbb{C})$}\label{sec:spectralgln}

Recall from \cite{N2} that the fibre of the Hitchin fibration for classical Higgs bundles is isomorphic to the Jacobian of a curve. Indeed, a classical Higgs bundle $(E,\Phi)$, or $GL(n,\C)$-Higgs bundle, has associated an $n$-fold cover $p:S\rightarrow \Sigma$, in the total space of $K$, with equation
\[\det (\eta-\Phi)=0, \]
 for $\eta$ the tautological section of $p^{*}K$, together with a line bundle $U$ on $S$ satisfying $p_{*}U=E$. Explicitly, $S$ has equation
\begin{equation}
\eta^n + a_{1}\eta^{n-1}+a_2 \eta^{n-2} + a_3 \eta^{n-3} + \cdots + a_n = 0,\label{eq.spec.gl}
\end{equation}
where $a_m\in H^{0}(\Sigma,K^m)$. Moreover, through Remark \ref{inv:gl} one can obtain a basis of invariant polynomials fixed by the induced involution, giving the coefficients $a_{i}$ via the Hitchin map.

The line bundle $U$ associated to a $GL(n,\mathbb{C})$-Higgs bundle $(E,\Phi)$  represents the eigenspaces of $\Phi$ and may be defined by the following exact sequence \cite{bobi}:
\begin{equation}\label{eq.specline}
\xymatrix{
0 \ar[r] & U \otimes p^* K^{1-n} \ar[r] &  p^*E \ar[rr]^(.4){\eta - p^*(\Phi)} & & p^*(E \otimes K)  \ar[r] & U \otimes p^*K \ar[r] & 0.
}
\end{equation}

Conversely given sections $a_i \in H^0(\Sigma , K^i)$ such that the associated curve $S$ defined by Equation (\ref{eq.spec.gl}) is smooth, and a line bundle $U$ on $S$, one has a corresponding rank $n$ stable Higgs bundle $(E,\Phi)$ by considering $E = p_*(U)$ and $\Phi$ the map obtained by pushing forward the tautological section $\eta : U \to U \otimes p^* K$.

As seen in Section \ref{sec:inro-Higgs}, we may view a Higgs bundle $(\overline{\partial}_A , \Phi)$ as consisting of a holomorphic structure $\overline{\partial}_A$ and a Higgs field $\Phi$, where $\Phi \in \Omega^{1,0}( \Sigma , \mathfrak{g}_c )$ is holomorphic with respect to $\overline{\partial}_A$. From Proposition \ref{ind.inv}, the action of the anti-holomorphic involution $f$ sends the pair $(\overline{\partial}_A,\Phi)$ to $f(\overline{\partial}_A,\Phi) = (f^*(\partial_A) , f^*(\Phi^*)).$ Hence, the spectral curve for $f(\overline{\partial}_A,\Phi)$ has equation
\begin{equation}\label{eq.spec.fsl}
{\rm det}(\eta - f^*(\Phi^{*})) = 0.
\end{equation}
 
The anti-involution $f$ naturally lifts to an anti-holomorphic involution $\tilde{f} : K \to K$ on the total space   by sending $y \in K_x$ to $f^*(\overline{y}) \in K_{f(x)}$.

\begin{defn} We denote by $\tilde{f}$ the natural lift of $f$ to all powers $K^m$ of $K$. \label{def.tilde.f}
\end{defn}

Let $p_K : K \to \Sigma$ be the projection from the total space of $K$. Then, the map $\tilde{f}$ can be lifted to a natural action $\tilde{f} : p_K^*(K^m) \to p_K^*(K^m)$ on the total space of $p_K^*(K^m)$. Moreover since $\tilde{f}^*(\eta) = \eta$, one has 
\begin{equation*}
\tilde{f}( {\rm det}(\eta- \Phi ) ) = {\rm det}( \tilde{f}(\eta) - f^*(\Phi^*) ) = {\rm det}( \eta - f^*(\Phi^*)),
\end{equation*}
and thus the action of $\tilde{f}$ on $K$ restricts to a bijection between the spectral curves for $(\overline{\partial}_A,\Phi)$ and $f(\overline{\partial}_A,\Phi)$. Let $(\overline{\partial}_A,\Phi)$ be a fixed point in the moduli space, so $f(\overline{\partial}_A,\Phi)$ is gauge equivalent to $(\overline{\partial}_A,\Phi)$ under a complex gauge transformation. In particular, since $(\overline{\partial}_A,\Phi)$ and $f(\overline{\partial}_A,\Phi)$ are in the same isomorphism class,  this requires the spectral curves to coincide.
\begin{prop} \label{prop.eigen.clas}
The spectral curves for $(\overline{\partial}_A,\Phi)$ and $f(\overline{\partial}_A,\Phi)$ coincide if and only if the coefficients $a_m$ in (\ref{eq.spec.gl}) satisfy the reality conditions
\begin{equation}
a_m = f^*(\overline{a_m}).\label{eq.real.cond.sl1}
\end{equation}

\end{prop}
\begin{proof}
Since $\Phi$ and $\Phi^*$ have conjugate eigenvalues we find from (\ref{eq.spec.fsl}) that the coefficient of $\lambda^{n-m}$ in the spectral curve for $f(A,\Phi)$ is given by $f^*(\overline{a_{m}})$. Hence, from (\ref{eq.spec.fsl}), the spectral curves for $(\overline{\partial}_A,\Phi)$ and $f(\overline{\partial}_A,\Phi)$ coincide if and only if $a_m$ satisfy $
a_m = f^*(\overline{a_m}).
$\end{proof}
\begin{rem}When  condition (\ref{eq.real.cond.sl1}) holds we have that $\tilde{f}$ restricts to an anti-holomorphic involution on the spectral curve $S$ covering $f$.
\end{rem}

Bearing in mind the spectral data associated to classical Higgs bundles, one can obtain a similar description of the spectral data associated to $f(\overline{\partial}_A,\Phi)$.

\begin{prop}\label{prop:data:classic}
For $U$ the eigen-line associated to a classical Higgs pair $(\overline{\partial}_A,\Phi)$, the spectral data on the curve $S$ for $f(\overline{\partial}_A,\Phi)$ is given by the line bundle 
\begin{eqnarray*}
\tilde{f}^*( \overline{U}^* \otimes p^{*}\overline{K}^{n-1}) = \tilde{f}^*(\overline{U}^*) \otimes p^{*} K^{n-1}.
\end{eqnarray*}
\end{prop}

\begin{proof}
Let $U'$ denote the spectral line associated to the pair $f(\overline{\partial}_A,\Phi) = (f^*(\partial_A) , f^*(\Phi^*))$. If $\overline{\partial}_A$ is a $\overline{\partial}$-operator on $E$ then $f^*(\partial_A)$ is the $\overline{\partial}$-operator naturally associated to $f^*(\overline{E}^*)$. To see this we can locally write $\overline{\partial}_A$ in the form $\overline{\partial}_A = \overline{\partial} + A$ with respect to a unitary frame. Then the corresponding complex structure on $f^*(\overline{E})$ is $\overline{\partial} + f^*(\overline{A})$, and the associated complex structure on $f^*(\overline{E}^*)$ is $\overline{\partial} - f^*(\overline{A^t})$. The claim follows since $\partial_A = \partial + \rho(A) = \partial - \overline{A^t}$. Similarly the induced Higgs field on $f^*(\overline{E}^*)$ is $f^*( \overline{\Phi^t}) = f^*(\Phi^*)$. To determine the spectral line associated to $f^*(\Phi^*)$ consider the exact sequence (\ref{eq.specline}) defining $U$. Dualising, pulling back by $\tilde{f}$, conjugating and tensoring by $p^*K$, we obtain the exact sequence
\begin{equation*}
\xymatrix{
0 \ar[r] & \tilde{f}^*(\overline{U}^*) \ar[r] &  p^*f^*(\overline{E}^*) \ar[rr]^(.43){\eta - p^*(\Phi^*)} & & p^*(f^*(\overline{E}^*) \otimes K)  \ar[r] & \tilde{f}^*(\overline{U}^*) \otimes p^*K^n \ar[r] & 0.
}
\end{equation*}
From this we see that the spectral line $U'$ for $f(\overline{\partial}_A,\Phi)$ is given by $U' = \tilde{f}^*(\overline{U}^*) \otimes p^*K^{n-1}$ as required.
 
\end{proof}


\begin{defn}\label{def:iota}
We denote by  $\iota : Pic(S) \to Pic(S)$  the natural involution defined by \[U\mapsto \tilde{f}^*( \overline{U}^* ) \otimes p^{*} K^{n-1}.\]
\end{defn}

 It is clear that $\iota$ is anti-holomorphic with respect to the natural complex structure on $Pic(S)$. We have thus established the following:

\begin{prop}
Let $(\overline{\partial}_A,\Phi)$ be a stable classical Higgs bundle with smooth spectral curve $S \to \Sigma$ and eigen-line $U\in Pic(S)$. Then $(\overline{\partial}_A,\Phi)$ is gauge equivalent to $f(\overline{\partial}_A,\Phi)$ if and only if 
\begin{itemize}
\item{$S$ is carried to itself under the natural lift $\tilde{f} : K \to K$ of $f$;}
\item{$U$ is fixed by the natural involution $\iota$ on $Pic(S)$.}
\end{itemize}
\end{prop}

\begin{defn}\label{def:iota0}
Define $\iota_0 : Jac(S) \to Jac(S)$ as the involution 
\[\iota_0: U \mapsto \tilde{f}^*(\overline{U}^*).\]
\end{defn}
The involution $\iota$ on $Jac(S)$ as in Definition \ref{def:iota}  can thus be expressed as
\[U\mapsto \iota_{0}(U)\otimes p^{*}K^{n-1},\]
and  changes degree according to
\begin{equation*}
{\rm deg}(\iota U) = -{\rm deg}(U) + n(n-1)(2g-2).
\end{equation*}
Thus only line bundles of degree $n(n-1)(g-1)$ can be fixed by $\iota$.  Let $Jac_d(S)$ denote the component of $Pic(S)$ consisting of line bundles of degree $d$ and write $Jac(S)$ for the Jacobian $Jac_0(S)$.
From the above analysis, any fixed point of  $\iota$ must lie in $Jac_{d}(S)$ for $d = n(n-1)(g-1)$, but we have not yet established the existence of any fixed points. For this, let $K^{1/2}$ be a theta characteristic corresponding to an $f$-invariant spin structure on $\Sigma$. Such spin structures exist by Proposition \ref{prop.fixedspin}, and by Proposition \ref{prop.actionspin} this corresponds to the existence of a theta characteristic $K^{1/2}$ such that $f^*(\overline{K}^{1/2}) \simeq K^{1/2}$. Now since $p^*K^{(n-1)/2}$ has degree $d = n(n-1)(g-1)$, tensoring by $p^* K^{(n-1)/2}$ defines a bijection $Jac(S) \to Jac_d(S)$. For $M$ a line bundle in $Jac(S)$, the action of $\iota$ on $M \otimes p^* K^{(n-1)/2}$ is
\begin{equation*}
\iota( M \otimes p^* K^{(n-1)/2} ) =  \iota_0(M) \otimes p^* K^{(1-n)/2} \otimes p^* K^{n-1} = \iota_0(M) \otimes p^* K^{(n-1)/2}.
\end{equation*}
Thus $M \otimes p^* K^{(n-1)/2}$ is fixed by $\iota$ if and only if $M$ is fixed by $\iota_0$.

\begin{lem}\label{lem.torusinvol}
Let $A$ be a complex torus of dimension $m$ and $\theta : A \to A$ an anti-holomorphic group automorphism. Then the fixed point set $A^\theta$ of $A$ is isomorphic to a product $T \times (\mathbb{Z}_2)^d$ of a real torus $T$ of dimension $m$ and a discrete group $(\mathbb{Z}_2)^d$, for some $0 \le d \le m$. Every component of $A^\theta$ contains a point of order $2$ and there are exactly $2^{m+d} $ points of order $2$ fixed by $\theta$.
\end{lem}
\begin{proof}
Since $\theta$ is an automorphism the fixed point set $A^\theta$ is a closed subgroup, hence a compact abelian subgroup. The identity component $A^\theta_0$ is therefore a torus which must have real dimension $m$ since $\theta$ is an anti-holomorphic involution. Let $H = A^\theta/A^\theta_0$ be the finite abelian group of components of $A^\theta$. The exact sequence $0 \to A^\theta_0 \to A^\theta \to H \to 0$ splits because $A^\theta_0$ is a divisible group. We claim that every element of $H$ is $2$-torsion. Indeed, if $x \in A^\theta$ then $x + x = x + \theta(x)$ lies in the identity component $A^\theta_0$ because it lies in the image of the connected group $A$ under the map $A \to A^\theta$ which sends a point $y \in A$ to $y + \theta(y)$. It remains only to show that every component of $A^\theta$ contains a point of order $2$. This follows since $A^\theta_0$ is a divisible group. The inequality $d \le m$ holds since $A$ has $2^{2m}$ points of order $2$. 
\end{proof}

\begin{prop}\label{prop:existence.gl}
The fixed point set $\CL_{GL(n,\mathbb{C})}$ of the involution $f$ on $\mathcal{M}_{GL(n,\mathbb{C})}$ is non-empty and contains smooth points of $\mathcal{M}_{GL(n,\mathbb{C})}$. A generic fibre of the restricted Hitchin fibration $\CL_{GL(n,\mathbb{C})} \to L$ is diffeormorphic to a product $T \times (\mathbb{Z}_2)^d$ of a torus of real dimension $1 + n^2(g-1)$ and a discrete group $(\mathbb{Z}_2)^d$ for some $d \ge 0$.
\end{prop}

\begin{proof}
A generic point in the base $L$ of the restricted Hitchin map $\CL_{GL(n,\mathbb{C})} \to L$ defines a smooth spectral curve $S \to \Sigma$ and the corresponding fibre is diffeomorphic to the fixed point set of the involution $\iota_0 : Jac(S) \to Jac(S)$, where $\iota_0(L) = \tilde{f}^*( \overline{L}^*)$. Thus $\iota_0$ is an anti-holomorphic group automorphism of $Jac(s)$ and we may apply Lemma \ref{lem.torusinvol}.
\end{proof}

Similar methods can be used to show that $\CL_{G_{c}}$ is non-empty for classical complex Lie groups  $G_{c}$. 
For completion, we shall show here how the method applies in the case of $G_{c}=SL(n,\C)$.


\subsection{Spectral data for $SL(n,\mathbb{C})$}\label{sec:spec:sl}
   An $SL(n,\mathbb{C})$-Higgs bundle, for $n\geq 2$, is a pair $(E,\Phi)$ where $E$ is a rank $n$ vector bundle with trivial determinant and $\Phi$ is a trace-free Higgs field. As usual we let $\overline{\partial}_A$ denote the $\overline{\partial}$-operator defining the holomorphic structure on $E$.

The spectral curve $p:S\rightarrow \Sigma$ associated to an $SL(n,\C)$-Higgs pair $(\overline{\partial}_A,\Phi)$  has equation
\begin{equation}
\eta^n + a_2 \eta^{n-2} + a_3 \eta^{n-3} + \cdots + a_n = 0,\label{eq.spec.sl}
\end{equation}
where $a_m$ is a holomorphic section in $H^{0}(\Sigma,K^m)$. From Proposition \ref{ind.inv}, the action of $f$ maps $(\overline{\partial}_A,\Phi) \mapsto f(\overline{\partial}_A,\Phi) = (f^*(\partial_A) , f^*(\Phi^*)).$ Thus, one can define the spectral curve for $f(\overline{\partial}_A,\Phi)$ which has equation ${\rm det}(\eta - f^*(\Phi^{*})) = 0$. 
%
%

From Proposition \ref{prop.eigen.clas}, the spectral curves for $(\overline{\partial}_A,\Phi)$ and $f(\overline{\partial}_A,\Phi)$ coincide if and only if the coefficients $a_m$ in (\ref{eq.spec.sl}) satisfy the reality conditions $a_m = f^*(\overline{a_m})$. In such a case, we have that $\tilde{f}$ restricts to an anti-holomorphic involution on the spectral curve $S$ covering $f$.

From  \cite[Section 2.2]{Hit07}, the spectral data in this case is given by an eigen-line bundle $U$ in the Jacobian of $S$ for which $U\otimes  p^{*}K^{(n-1)/2}$  lies in the Prym variety ${\rm Prym}(S,\Sigma)$. Recall that ${\rm Prym}(S,\Sigma)$ is defined by the exact sequence
\begin{equation*}\xymatrix{
1 \ar[r]& {\rm Prym}(S,\Sigma) \ar[r]& {\rm Jac}(S) \ar[r]^{Nm}& {\rm Jac}(\Sigma) \ar[r]&1,}
\end{equation*}
where the norm map $Nm : {\rm Jac}(S) \to {\rm Jac}(\Sigma)$ is defined by sending a divisor $D = \sum_i s_i$ on $S$ to the divisor $Nm(D) = \sum_i p(s_i)$ on $\Sigma$. Note that ${\rm Prym}(S,\Sigma)$ in this situation is connected, hence a complex torus which is in fact an abelian variety.

Through Proposition \ref{prop:data:classic},   the spectral data associated to the pair $f(\overline{\partial}_A,\Phi)$ is given by  the line bundle 
\begin{eqnarray*}
\tilde{f}^*( \overline{U}^* \otimes p^{*}\overline{K}^{n-1}) = \tilde{f}^*(\overline{U}^*) \otimes p^{*} K^{n-1}.
\end{eqnarray*}
 %
 
To proceed to $SL(n,\mathbb{C})$-bundles we must understand how the fixed point set of $\iota_0$ meets the Prym variety.

\begin{prop} The action of $\iota_0$ on $Jac(S)$ preserves the Prym variety.  \label{prop:iotaPrymSl}
\end{prop}
\begin{proof}
Recall that the norm map $Nm $ may also be defined as the dual map of the pullback $p^* : Jac(\Sigma) \to Jac(S)$. Let $\iota'_0 : Jac(\Sigma) \to Jac(\Sigma)$ be defined by $\iota'_0(L) = f^*(\overline{L}^*)$, then clearly $p^* \iota'_0 = \iota_0 p^* : Jac(\Sigma) \to Jac(S)$. It follows that $\iota'_0 \circ Nm = Nm \circ \iota_0$, in particular $\iota_0$ preserves the Prym variety, which is given by the kernel of $Nm$.
\end{proof}

Applying Lemma \ref{lem.torusinvol} we obtain the following:
\begin{prop}\label{prop:slfib}
The generic fibres of the restricted Hitchin fibration $\CL_{SL(n,\mathbb{C})} \to L$ are products $T \times (\mathbb{Z}_2)^d$ of a real torus $T$ of dimension $(n^2-1)(g-1)$ and a discrete group $(\mathbb{Z}_2)^d$, where the number of points of order $2$ in ${\rm Prym}(S,\Sigma)$ fixed by $\iota_0$ is $2^{(n^2-1)(g-1) + d}$.
\end{prop}

The above study establishes the existence of fixed points for the action of $f$ on the moduli space of $SL(n,\mathbb{C})$-Higgs bundles. Moreover, it  confirms the picture given in Section \ref{sec:hitFib}, where the fixed point set is shown to be a real integrable system. 



\begin{rem}
One should note that the moduli space of $SU(p,p)$-Higgs bundles can be seen inside the moduli space of $SL(2p,\mathbb{C})$-Higgs bundles. In particular from \cite[Chapter 6]{thesis}, the spectral data for this subspace corresponds to line bundles $L$ of fixed determinant which are preserved by the natural involution $\sigma:\eta \mapsto -\eta$ on the corresponding spectral curve. Hence, by looking at the relation between $\sigma$ and $\tilde f$ one can understand the fixed point set of $f$ in the moduli space $\mathcal{M}_{SU(p,p)}$ ( equivalently, in $\mathcal{M}_{U(p,p)}$).
\end{rem}

Applying the same methods for the case of  $G_{c}=PGL(n,\C), Sp(2n,\C), SO(2n+1,\C), $ and $SO(2n,\C)$, one can see through the spectral data introduced in \cite{Hit07}  that the fixed point set $\CL_{G_{c}}$ is non-empty and contains smooth points. Moreover the generic fibre of the Hitchin map is the set of real points in an abelian variety, diffeomorphic to $2^d$ torus components for some $d$. In the following section, we shall study the topology of the $(A,B,A)$-branes $\CL_{G_{c}}$.
 %


\section{Connectivity of the brane $\CL_{G_{c}}$}\label{sec:conn}\label{sec:con}

In order to study connectivity of the $(A,B,A)$-brane $\CL_{G_{c}}$ in the moduli space $\CM_{G_{c}}$ of $G_{c}$-Higgs bundles, we consider the fibres of the restricted Hitchin fibration
\[h|_{\CL_{G_{c}}}: \CL_{G_{c}}\rightarrow L\]
In particular, we shall see that the number of connected components of the fibres depends on the invariants $(n,a)$ associated to the Riemann surface and the spectral curves, as introduced in Section \ref{sec:classif} following \cite{GH81}. 


\subsection{Connectivity for classical Higgs bundles}\label{subsec:concla}

\begin{proposition}\label{prop:numbergl}
Let $S \to \Sigma$ be a smooth $GL(n,\C)$ spectral curve, $n_S$  the number of fixed point components of the lifted involution $\tilde{f} : S \to S$, and $g_S = 1 + n^2(g-1)$ the genus of $S$. Then the number of connected components of the corresponding fibre of $$h|_{\CL_{GL(n,\C)}}: \CL_{GL(n,\C)}\rightarrow L$$ is $2^{n_S-1}$ if $n_S > 0$. If $n_S = 0$  the number of components is $1$ if $g_S$ is even, and $2$ if $g_S$ is odd.\end{proposition}
 
\begin{proof} 
Recall that  the involution $\iota_0 : Jac(S) \to Jac(S)$ is given by $\iota_0(L) = \tilde{f}^*(\overline{L}^*)$, and that the fibre of the Hitchin map corresponding to the spectral curve $S$ is given by the fixed point set of $\iota_0$. Thus we must determine the number of components of the fixed point set of $\iota_0$. From Lemma \ref{lem.torusinvol} we see that the number of components of the fixed point set of $\iota_0$ is the same as the number of components of the involution $\theta : Jac(S) \to Jac(S)$, where $\theta(L) = \tilde{f}^*(\overline{L})$. The number of fixed points of the involution $\theta$ is studied in \cite[Propositions 3.2-3.3]{GH81}, where it is determined that the number of components is $2^{n_S-1}$ if $n_S > 0$. In the case that $n_S = 0$ the number of fixed points is shown to be $1$ or $2$ depending on whether the genus of the curve is even or odd.
\end{proof}

It remains to determine how $n_S$ depends on the pair $(\Sigma , f)$ and the coefficients $a_i \in H^0(\Sigma , K^i)$ defining the spectral curve $S$. As the rank $n$ increases this quickly becomes difficult so we restrict attention to the rank $2$ case in Section \ref{sec:rank2case}.


The geometry and topology of $\CL_{G_{c}}$ for other groups can  also be  studied through the spectral data associated to elements in $\CL_{G_{c}}$. For this, one has to consider the spectral data for $G_{c}$-Higgs bundles which is fixed by the induced involution, bearing in mind the study of the action done in Section \ref{sec:indprin}. 
 In the following subsection we  look at an example of how the invariants $(n,a)$ can be obtained for $GL(2,\C)$ and $SL(2,\C)$-Higgs bundles.

\begin{remark}
In the case of  Accola-Maclachlan  and Kulkarni surfaces, the invariants $(n,a)$ are determined in \cite{acc} for any genus $g$. Thus, the study of the invariants for the ramified covering $S$ in this case gets simplified, and connectivity of $\CL_{GL(n,\C)}$  for these Riemann surfaces can be determined through \cite{GH81}. For hyperelliptic surfaces, a study of the invariants associated to $f$ has been done in \cite{acc2}, and extended to our case in  Appendix \ref{sec:clas}.
\end{remark}


\subsection{Rank $2$-case}\label{sec:rank2case}

In order to understand the invariants which characterise the connected components of the real points in the fibres of the classical Hitchin fibration for any Riemann surface, i.e., real points in ${\rm Jac}(S)$, we shall consider rank $2$ Higgs bundles. In this case the Hitchin fibration is given by
\[h:(E,\Phi)\mapsto (a_{1},a_{2})\in H^{0}(\Sigma, K)\oplus H^{0}(\Sigma, K^{2}),\]
and the corresponding 2-fold cover $p:S\rightarrow \Sigma$ has genus $g_{S}=1+4(g-1)$ and equation
$\eta^{2}+a_{1}\eta+a_{2}=0.$ 
Since we can re-write this equation as $(\eta + a_1/2)^2 + (a_2 - a_1^2/4) = 0$ there will be no loss in generality in assuming $a_1 = 0$. The spectral curve $S$ is therefore a double cover given by
\begin{eqnarray*}S= \{ (\eta, z) \in K ~|~ \eta^{2}=q(z)\}\end{eqnarray*}
for $q = -a_2 \in H^{0}(\Sigma, K^{2})$. Generically $q$ has $4g-4$ different zeros, which give the ramification points of the smooth curve $S$. The reality condition on the defining equation for $S$ is simply that $f^*(\overline{q}) = q$, and  in particular $f$ must act on the zero set of $q$.

The anti-linear involution $f$ on the Riemann surface $\Sigma$ induces an involution $\tilde{f}:S\rightarrow  S$, which from Section \ref{sec:hitFib} acts by 
\begin{eqnarray}\begin{aligned}
 (\eta,z)\mapsto  (\overline{f^{*}( \eta)}, f(z))
\end{aligned}\label{invodefi}
\end{eqnarray}
The study of the invariant $n_{S}$ depends on how $f$ acts on the zero set of $q$. First note that since $\tilde{f}$ covers $f$, any fixed component of $\tilde{f}$ must lie over a fixed component of $f$. Therefore to determine $n_S$ we need only determine how many fixed components of $\tilde{f}$ lie over each fixed component of $f$. Given a fixed component of $f$ there are two cases to consider depending on whether or not $q$ has zeros along the component.

\begin{proposition}\label{prop:nozeros}
If $S^{1}$ is a fixed circle component of $f$ in $\Sigma$ and the differential $q$ does not have any zeros on $S^{1}$, then $p^{-1}(S^{1})=S^{1}\cup S^{1}$ and
\begin{itemize}
\item the two $S^{1}$ factors in $S$ are fixed by $\tilde f$ if $q_{1}>0$ on $S^{1}$, 
\item the two $S^{1}$ factors in $S$ are exchanged by $\tilde f$ if $q_{1}<0$ on $S^{1}$.
\end{itemize}
Moreover, the differential may be written in a local coordinate $z$ as $q(z)=q_{1}(z)(dz)^{2}$, where $f(z) = \overline{z}$ and $q_1(\overline{z}) = \overline{q_1(z)}$. 
\end{proposition}
\begin{proof}
Let $x$ be a fixed point of $f$ such that $q(x)\neq 0$. By choosing an appropriate  local holomorphic coordinate $z$, as seen before we can write $f(z)=\bar z$ in a neighbourhood of $x$, corresponding to $z=0$. In local coordinates   the differential can be expressed as 
\[q(z)=q_{1}(z)(dz)^{2}~{~\rm~for~}~ q_{1}(0)\neq 0,\]
 and by rescaling $z$ by a real factor if necessary, we may assume that $q_{1}(0)=\pm1$. The curve $S$ is thus given by
 \[\eta^{2}=q_1(z)(dz)^{2},\]
 and so at $x$, i.e., $z=0$, we have
 \begin{eqnarray*}
 \eta=\left\{
 \begin{array}
 {l l}
 \pm dz& q_{1}(0)=1;\\
  \pm i ~dz& q_{1}(0)=-1.
 \end{array}
 \right.
  \end{eqnarray*}
 In the case of $q_{1}(0)=1$ the points lying over $x$ are $(\pm dz, 0)$, and by (\ref{invodefi}) they  are both fixed by $\tilde{f}$.
 In the case of $q_{1}(0)=-1$ the points over $x$ are $(\pm i~dz, 0)$, which are interchanged by $\tilde f$.
 
Consider now $S^{1}\subset \Sigma$ a component of the fixed points of $f$ in $\Sigma$ which does not contain zeros of $q$. As above, for $x\in S^{1}$ we can choose a  holomorphic local coordinate $z$ such that $f(z)=\bar z$, and such that the point $x$ corresponds to $z=0$. 
 For any other such coordinate $w$, one can write
 \[w=a_{0}z+a_{1}z^{2}+ \ldots ~{\rm~for~}~ a_{0}\neq 0, ~a_{i} \in \R.\]
  Then $dw|_{x}=a_{0}dz|_x$,  and this defines a real subbundle $K|_{S^1}(\mathbb{R})$ of the restriction $K|_{S^1}$ of $K$ to the fixed circle.
Note that since  we can find a non-vanishing vector field $X \in T\Sigma|_{S^1}$ tangent to $S^1$, the subbundle $K|_{S^1}(\mathbb{R})$ is trivial.  Therefore, since $K|_{S^1}(\mathbb{R})$ is a trivial real bundle and $\pm \eta$ are non-vanishing along the circle, then the double cover $p^{-1}(S^1) \to S^1$ defined by $\eta^2 = q$ must be a trivial double cover $p^{-1}(S^{1})=S^{1}\cup S^{1}$.
\end{proof}
 
We shall now consider the case of a fixed circle $S^{1}\subset \Sigma$ which contains one or more zeros of the quadratic differential $q$.  Let $x\in S^{1}$ such that $q(x)=0$, and choosing local holomorphic coordinates as before we set $f(z)=\bar z$, and $z=0$ for the point $x$.  As in the above proof, we can write
$q(z)=q_{1}(z)(dz)^{2},$
where $q_{1}(\bar z )=\overline{q_{1}(z)}$ and $q_{1}(0)=0$. Further, since we are assuming that $q$ has only simple zeros we have $dq_{1}(z)\neq 0$ at $z=0$. Hence, we can set $q_{1}(z)$ as a local real coordinate
$w=q_{1}(z).$
 In terms of this coordinate on $\Sigma$, the double cover 
$
p:S\rightarrow \Sigma $ given by $ (\eta dz,z)\mapsto z$ can be written as
 \[\eta \mapsto \eta^{2}= q_{1}(z)=w.\]
 Recall that $\tilde f(\eta)=\bar \eta$ and $f (w) =\bar w$. If $\eta$ is real one has $\tilde f (\eta)=\eta$ and if $\eta$ is imaginary we get $\tilde{f}(\eta)=-\eta$. Locally the fixed points of $f$ are the points where $w$ is real and so 
 \begin{itemize}
 \item for $q_1 = w<0$, we have that $\pm \eta \in i\R$, and $\tilde{f}$ exchanges $\eta$ and $-\eta$,
 \item for $q_1 = w>0$, we have that $\pm \eta \in \R$, and $\tilde{f}$ fixes $\pm \eta$.
 \end{itemize}
 From the above analysis one has the following proposition:

\begin{proposition}\label{prop:graph}
Let $S^1 \subset \Sigma$ be a fixed component of $f$ on which $q$ has at least one zero. Then $q$ has an even number $2k$ of zeros which divides the circle into $2k$ segments over which we have $q \ge 0$ and $q \le 0$ alternately. For each of the $k$ segments where $q \ge 0$, the inverse image of the segment in $S$ is a circle fixed by $\tilde{f}$. For each of the $k$ segments where $q \le 0$, the inverse image of the segment is a circle such that $\tilde{f}$ acts as a reflection. In particular the inverse image $p^{-1}(S)$ contains exactly $k$ circles fixed by $\tilde{f}$.
\end{proposition}

It is possible to strengthen this result to a precise description of $S$ in a neighborhood of the inverse image $p^{-1}(S^1)$. The inverse image $p^{-1}(S^1)$ is the graph obtained by taking $2k$ circles and joining them in a chain which has Euler characteristic $2k-4k = -2k$. We can enlarge each circle of $p^{-1}(S^1)$ to a tubular neighbourhood in $S$ such that their union is a neighbourhood of $p^{-1}(S^1)$ diffeomorphic to a surface of genus $k-1$ with $4$ points removed. Combining Propositions \ref{prop:nozeros} and \ref{prop:graph} we obtain:

\begin{thm}\label{thm:numbergl2}
Let $q \in H^0(\Sigma , K^2)$ be a quadratic differential which has only simple zeros, and such that $f^*(\overline{q}) = q$, and  let $S$ be the spectral curve given by $\eta^2 = q$. Further, let $n_+$ be the number of fixed components of $f$ on which $q$ is non-vanishing and positive, and $u$ the number of zeros of $q$ which are fixed by $f$. Then the  number of components of the induced involution $\iota_0 : Jac(S) \to Jac(S)$ is $2^d$, where $d = 2n_+ + u/2 - 1$ if $2n_+ + u/2 > 0$, and $d = 1$ otherwise.
\end{thm}

It is natural to ask which pairs $(n_+,u/2)$ can occur for a given pair $(\Sigma , f)$ with associated invariants $(n,a)$. Note that in particular one has $0 \le n_+ \le n$ and $0 \le u/2 \le 2g-2$. Moreover if $u > 0$ then $n_+ < n$, and if $n = 0$ then $u = 0$. Even with these constraints there are a large number of possibilities that can occur. For example in the genus $g=2$ case the invariants $(n,a,n_+,u/2)$ subject to the above constraints together with the constraints on $(n,a)$ given by Proposition \ref{pro:topcond} give rise to a total of $26$ possible cases. To determine which of these possible cases actually occur one can use the explicit description of the space of quadratic differentials on a genus $2$ hyperelliptic curve given in Appendix \ref{sec:hyper}. Remarkably it turns out that all but one of the $26$ cases can actually be realized, the exception being the case $(n,a,n_+,u/2) = (1,0,0,1)$.


%
%
%
%
%
%
%
%

\subsection{Connectivity for $SL(2,\C)$-Higgs bundles}\label{sec:connsl2}
As mentioned before, in the case of $SL(2,\C)$-Higgs bundles the fibres of the Hitchin fibration are Prym varieties ${\rm Prym}(S,\Sigma)$, for $S$ the spectral curve defined by the point in the Hitchin base. Hence, in this case one is interested in the number of connected components of the intersection of ${\rm Prym}(S,\Sigma)$ with the real points in the Jacobian of $S$. 

Let $q \in H^0(\Sigma , K^2)$ be a quadratic differential with simple zeros and satisfying $f^*(q) = \overline{q}$. In particular $f$ preserves set of $4g-4$ branch points of $q$. For $p : S \to \Sigma$  the branched double cover of degree $2$ defined the equation $\eta^2 = q$,  its genus is $g_{S}=1 + 4(g-1)$. The natural action of $f$ on $K$ defines a lift $\tilde{f} : S \to S$ of $f$. Denote by $\sigma : S \to S$ the involution which exchanges sheets, then $\tilde{f}$ and $\sigma$ commute.

Since $S \to \Sigma$ is a double cover, the Prym variety ${\rm Prym}(S,\Sigma)$ may be defined as those line bundles $L$ satisfying $\sigma^* L \simeq L^*$,  and is a complex torus of dimension $3g-3$. We are interested in the set of elements $L \in {\rm Prym}(S,\Sigma)$ which satisfy the reality condition $\tilde{f}^*(\overline{L}^*) = L$. Since this set is the product of a real torus of dimension $3g-3$ and a discrete group $(\mathbb{Z}_2)^d$, the fixed point set has $2^d$ components for some value of $d$.
The strategy we will employ to find $d$ is to look for points of order $2$ in ${\rm Prym}(S,\Sigma)$ satisfying the reality condition. Let $P$ denote the set of such points, which is a subgroup of ${\rm Prym}(S,\Sigma)$. From Lemma \ref{lem.torusinvol},  the set $P$ is isomorphic to $(\mathbb{Z}_2)^{3g-3+d}$, and thus it will suffice to determine the order of $P$.

By construction, elements of $P$ are line bundles $L$ on $S$ satisfying the following three conditions: $L^2$ is trivial, $\sigma^*(L) = L$ and $\tilde{f}^*(\overline{L}) = L$. Since $L$ has order $2$, it is equivalent to view $L$ as a flat line bundle with transition functions in $\mathbb{Z}_2$. Then as flat line bundles $L \simeq \overline{L} \simeq L^*$. Thus the set $P$ corresponds to elements of $H^1(S , \mathbb{Z}_2)$ which are invariant under the actions of $\sigma$ and $\tilde{f}$. Let $H^1(S,\mathbb{Z}_2)^\sigma$ denote the subgroup of elements of $H^1(S,\mathbb{Z}_2)$ fixed by $\sigma$. We will first determine this group and then find the subgroup of elements fixed by $\tilde{f}$.

Let $b = 4g-4$ be the number of branch points of $p$, and $B = \{ x_1 , \dots , x_b \}$ be the set of branch points and $\tilde{x}_i = p^{-1}(x_i)$. For $\Sigma' = \Sigma \setminus B$ and $S' = S \setminus p^{-1}(B)$, there is a  commutative diagram
\begin{equation*}\xymatrix{
S' \ar[r]^{i_S} \ar[d]^p & S \ar[d]^p \\
\Sigma' \ar[r]^{i_\Sigma} & \Sigma
}
\end{equation*}
Moreover $p : S' \to \Sigma'$ is a degree $2$ covering space, so we obtain an exact sequence 
\begin{equation}\label{equ:coveringseq}\xymatrix{
1 \ar[r]&\pi_1(S',\tilde{x}) \ar[r]& \pi_1(\Sigma',x) \ar[r]^-{\omega} & \mathbb{Z}_2 \ar[r]& 1,}
\end{equation}
where $\tilde{x} \in S'$ and $x = p(\tilde{x}) \in \Sigma'$. The map $\omega : \pi_1(\Sigma',x) \to \mathbb{Z}_2$ is the cohomology class $\omega \in H^1(\Sigma' , \mathbb{Z}_2)$ corresponding to the double cover $p : S' \to \Sigma'$.

\begin{prop}\label{prop:numbersl2a}
There is an exact sequence
\begin{equation*}\xymatrix{
0 \ar[r]&\mathbb{Z}_2 \ar[r]& H^1(\Sigma' , \mathbb{Z}_2) \ar[r]^-{p^{*}}& i_S^*( H^1(S,\mathbb{Z}_2)^\sigma) \ar[r]& 0,}
\end{equation*}
where the kernel of $p^*$ is generated by $\omega$. Moreover, $i_S^* : H^1(S,\mathbb{Z}_2) \to H^1(S' , \mathbb{Z}_2)$ is injective and thus one has an isomorphism
$H^1(S,\mathbb{Z}_2)^\sigma \simeq H^1(\Sigma' , \mathbb{Z}_2) / \langle \omega \rangle.
$\end{prop}
\begin{proof}
The Mayer-Vietoris sequence applied to $\Sigma$ and $S$ gives a commutative diagram with exact rows
\begin{equation*}\xymatrix{
0 \ar[r] & H^1(S,\mathbb{Z}_2) \ar[r]^{i_S^*} & H^1(S' , \mathbb{Z}_2) \ar[r] & \mathbb{Z}_2^b \ar[r]^-{\delta} & H^2(S,\mathbb{Z}_2) \ar[r] & 0 \\
0 \ar[r] & H^1(\Sigma,\mathbb{Z}_2) \ar[r]^{i_\Sigma^*} \ar[u]^{p^*} & H^1(\Sigma' , \mathbb{Z}_2) \ar[r] \ar[u]^{p^*} 
& \mathbb{Z}_2^b \ar[r]^-{\delta} \ar[u]^{0} & H^2(\Sigma,\mathbb{Z}_2) \ar[r] \ar[u]^{0} & 0
}
\end{equation*}
This shows that $i_S^* : H^1(S,\mathbb{Z}_2) \to H^1(S' , \mathbb{Z}_2)$ is injective and that the image of the map $p^* : H^1(\Sigma' , \mathbb{Z}_2) \to H^1(S',\mathbb{Z}_2)$ is contained in the image of $i_S^*$. Clearly also the image of $p^*$ is $\sigma$-invariant since $p \sigma = p$, so $p^*( H^1(\Sigma',\mathbb{Z}_2)) \subseteq i_S^*( H^1(S,\mathbb{Z}_2)^\sigma)$. From the exact sequence (\ref{equ:coveringseq}) we see that the kernel of $p^* : H^1(\Sigma' , \mathbb{Z}_2) \to i_S^*( H^1(S,\mathbb{Z}_2)^\sigma)$ is $\langle \omega \rangle = \mathbb{Z}_2$, so it remains to show surjectivity of $p^*$.

Let  $\tilde{x}$ be in a small $\sigma$-invariant disc $D \subset S$ containing a single branch point $\tilde{x}_1$, such that  $\tilde{x} \neq \tilde{x}_1$. Let $\gamma$ be a path in $D \setminus \{ \tilde{x}_1 \}$ from $\tilde{x}$ to $\sigma(\tilde{x})$. We may also choose $D$ small enough so that the restriction $p|_D$ has the form $z \mapsto z^2$.
For $\hat{\sigma} : \pi_1(S,\tilde{x}) \to \pi_1(S,\tilde{x})$ given by $\hat{\sigma}(\alpha) = \gamma . \sigma(\alpha) . \gamma^{-1}$,  the map  $\hat{\sigma}$ is an automorphism of   $\pi_1(S,\tilde{x})$ which induces the pullback action on cohomologies  $\sigma^* : H^1(S,\mathbb{Z}_2) \to H^1(S,\mathbb{Z}_2)$ under the identification $H^1(S,\mathbb{Z}_2) = {\rm Hom}(\pi_1(S,\tilde{x}) , \mathbb{Z}_2)$.

For $\rho : \pi_1(S) \to \mathbb{Z}_2$ a homomorphism which is $\sigma^*$-invariant, one has  $\rho \circ \hat{\sigma} = \rho$. Moreover there is a representation $\tau : \pi_1(\Sigma') \to \mathbb{Z}_2$ such that $p^*(\tau) = i_S^*(\rho)$. Indeed, consider $\alpha$ a loop in $\Sigma'$ based at $x$. Then exactly one of the loops $\alpha$, $p(\gamma).\alpha$ lifts to a loop in $S'$ based at $\tilde{x}$. Let $\hat{\alpha}$ be this lift and set $\tau(\alpha) = \rho( \hat{\alpha})$. We claim that $\tau$ is a homomorphism. This follows from the fact that $\rho$ is $\hat{\sigma}$-invariant, and since $\rho( \gamma . \sigma(\gamma) ) = 1$ because $\gamma. \sigma(\gamma)$ is null-homotopic in $S$. By construction it is immediate that $p^*(\tau) = i_S^*(\rho)$.
\end{proof}

To find the group of $\tilde{f}$-invariant elements of $H^1(S,\mathbb{Z}_2)^\sigma$ it is equivalent to find the $f$-invariant elements of $H^1(\Sigma',\mathbb{Z}_2)/\langle \omega \rangle$, since we have the commutative diagram
\begin{equation*}\xymatrix{
\mathbb{Z}_2 \ar[r] & H^1(\Sigma' , \mathbb{Z}_2) \ar[r]^-{p^*} & i_S^*( H^1(S,\mathbb{Z}_2)^\sigma) \\
\mathbb{Z}_2 \ar[r] \ar[u]^{id} & H^1(\Sigma' , \mathbb{Z}_2) \ar[r]^-{p^*} \ar[u]^{f^*} & i_S^*( H^1(S,\mathbb{Z}_2)^\sigma) \ar[u]^{\tilde{f}^*}.
}
\end{equation*}
 For $A = H^1(\Sigma',\mathbb{Z}_2)/\langle \omega \rangle$, the involution $f$ acts on $A$ and we are after the group $A^f$ of $f$-invariant elements of $A$. Let $A^* = {\rm Hom}(A,\mathbb{Z}_2)$. Then since $A$ is a $\mathbb{Z}_2$-vector space,  $A$ and $A^*$ are isomorphic, and the fixed point subspace $A^f$ is isomorphic to the subspace $(A^*)^f$ of $A^*$ fixed by the dual action of $f$ on $A^*$. The action $f_* : A^* \to A^*$ fits into a commutative diagram with exact rows
\begin{equation*}\xymatrix{
0 \ar[r] & A^* \ar[r] & H_1(\Sigma' , \mathbb{Z}_2) \ar[r]^-{\omega} & \mathbb{Z}_2 \ar[r] & 0 \\
0 \ar[r] & A^* \ar[r] \ar[u]^{f_*} & H_1(\Sigma' , \mathbb{Z}_2) \ar[r]^-{\omega} \ar[u]^{f_*} & \mathbb{Z}_2 \ar[r] \ar[u]^{id} & 0.
}
\end{equation*}
Thus to determine $(A^*)^f$ it is sufficient to determine the group $H_1(\Sigma' , \mathbb{Z}_2)$, the action of  the map $f_* : H_1(\Sigma' , \mathbb{Z}_2) \to H_1(\Sigma' , \mathbb{Z}_2)$ and the homomorphism $\omega : H_1(\Sigma' , \mathbb{Z}_2) \to \mathbb{Z}_2$.
Note that the Mayer-Vietoris sequence applied to $\Sigma$ gives an exact sequence
\begin{equation*}\xymatrix{
0 \ar[r] & H_2(\Sigma,\mathbb{Z}_2) \ar[r]^-{\partial} & \mathbb{Z}_2^b \ar[r] & H_1(\Sigma' , \mathbb{Z}_2) \ar[r]^{(i_\Sigma)_*} & H_1(\Sigma , \mathbb{Z}_2) \ar[r] & 0.
}
\end{equation*}
The above $\mathbb{Z}_2^b$ group corresponds to cycles around each of the $b$ branch points, and the boundary $\partial : H^2(\Sigma,\mathbb{Z}_2) = \mathbb{Z}_2 \to \mathbb{Z}_2^b$ is the diagonal $\mathbb{Z}_2 \to \mathbb{Z}_2^b$. In particular, $H_1(\Sigma' , \mathbb{Z}_2)$ is a $\mathbb{Z}_2$-vector space of dimension $2g + b - 1 = 6g-5$.

\begin{thm}\label{thm:numbersl2b}
Suppose that at least one branch point of $p : S \to \Sigma$ is fixed by $f$. Then the fixed point set of the action of $\tilde{f}$ on the Prym variety ${\rm Prym}(S,\Sigma)$ has $2^{n_0 + u/2 - 1}$ connected components, where $n_0$ is the number of fixed components of $f : \Sigma \to \Sigma$ which do not contain branch points, and $u$ is the number of branch points which are fixed by $f$.
\end{thm}
\begin{proof}
Let $A^*$ be the kernel of $\omega : H_1(\Sigma' , \mathbb{Z}_2) \to \mathbb{Z}_2$ and $(A^*)^f$ the subspace of $A^*$ fixed by $f_*$. We must show that $(A^*)^f$ has dimension $3g-3 + n_0 + u/2 - 1$. Consider a cycle $D \in H_1(\Sigma' , \mathbb{Z}_2)$ around a branch point which is fixed by $f$. Then $f_*(D) = D$ and $\omega(D) = 1$. Therefore it suffices to show that the kernel of $\theta = f_* - 1 : H_1(\Sigma',\mathbb{Z}_2) \to H_1(\Sigma',\mathbb{Z}_2)$ has dimension $3g-3+n_0 + u/2$.

Recall that there are $b = 4g-4$ branch points and write $b = 2t + u$, where $t$ is the number of pairs of branch points exchanged by $f$, and $u$ is the number of fixed branch points. For $i = 1, \dots , t$ let $C_i,C'_i$ be cycles around a pair of branch points exchanged by $f$. Further, for  $j = 1, \dots, u$ denote by  $D_j$ cycles around the fixed branch points. From the Mayer-Vietoris sequence we have a single relation between these cycles $\sum_{i=1}^t C_i + C'_i + \sum_{i=1}^u D_j = 0.$ These cycles span a $b-1$-dimensional subspace of $H_1(\Sigma',\mathbb{Z}_2)$ and we have $\theta( C_i ) = \theta(C'_i) = C_i + C'_i$, $\theta(D_j) = 0$.

Let $a_1, \dots , a_{2g} \in H_1(\Sigma , \mathbb{Z}_2)$ be a basis of cycles in $\Sigma$. We may assume that the cycles do not touch the branch points so that $a_1,\dots , a_{2g}$ are also cycles in $\Sigma'$. Together with the cycles $C_i,C_i',D_j$ we have a generating set for $H_1(\Sigma',\mathbb{Z}_2)$ satisfying  $\sum_{i=1}^t C_i + C'_i + \sum_{i=1}^u D_j = 0$.
To proceed we consider three subcases depending on the topology of $(\Sigma , f)$. Let $(n,a)$ be the topological invariants associated to the anti-holomorphic involution $f$.

{\it Case 1: $a = 0$.} Write $g = 2s+r$, $n = r+1$. We have a generating set for $H_1(\Sigma',\mathbb{Z}_2)$ consisting of $A_i,B_i,A'_i,B'_i, A''_j,B''_j , C_k,C'_k, D_{l}$ for $1 \le i \le s$,  $ 1 \le j \le r$,  $ 0 \le k \le t$,  and   for $ 1 \le l \le u$ with one relation $\sum_{k=1}^t C_k + C'_k + \sum_{l=1}^u D_l = 0$. The cycles $A_i,A'_i$ are interchanged by $f_*$ and similarly for $B_i,B'_i$. The cycles $A''_i$ correspond to a choice of $r = n-1$ of the fixed components of $f$, perturbed slightly so as to avoid the branch points. Let $\sum_{h \in A''_j} D_h$ denote the sum of cycles $D_h$ where $D_h$ is a cycle around a branch point which lies in the fixed component of $f$ corresponding to $A''_j$. Finally the $B''_j$ are cycles that cross two of the fixed components of $f$. We can choose the $A''_j,B''_j$ so that $\theta(A''_j) = \sum_{h \in A''_j} D_h$ and $\theta(B''_j) = 0$. Now it is a straightforward computation to show that ${\rm Ker}(\theta)$ has dimension $2s+r + (n_0-1) + t + u = 3g-3 + n_0 + u/2$ as required.

{\it Case 2: $a=1$, and $g-n$ even.} Write $g = 2s+r$, $n=r$. We have a generating set for $H_1(\Sigma',\mathbb{Z}_2)$ consisting of $A_i,B_i,A'_i,B'_i, A''_j,B''_j,C_k,C'_k,D_{l}$  with one relation as before. As previously, $f_*$ exchanges the pairs $A_i,A'_i$ and $B_i,B'_i$, and the $A''_j$ correspond to the fixed components of $f$. We again have $\theta(A''_j) = \sum_{h \in A''_j} D_h$, but now we have $\theta(B''_j) = \sum_{j=1}^r A''_j + \sum_{k=1}^t C_k$. Again we find ${\rm Ker}(\theta)$ has dimension $3g-3 + n_0 + u/2$.

{\it Case 3: $a=1$, and $g-n$ odd. }Write $g = 2s+r+1$, $n=r$. We have a generating set for $H_1(\Sigma',\mathbb{Z}_2)$ consisting of $X,Y$, and $A_i,B_i,A'_i,B'_i, A''_j,B''_j, C_k, C'_k, D_l,$ with one relation as before. We have that $f_*(X) = X$ and $f_*(Y) = Y + \sum_{j=1}^r A''_j$. The pairs $A_i,A'_i$ and $B_i,B'_i$ are exchanged by $f_*$. The $A''_j$ again correspond to fixed components of the involution  $f$ and satisfy $\theta(A''_j) = \sum_{h \in A''_j} D_h$. Finally we have $\theta(B''_j) = Y$. In this case we again verify that ${\rm Ker}(\theta)$ has dimension $3g-3 + n_0 + u/2$.
\end{proof}


\section{Real and quaternionic bundles}\label{sec:realquat}
We shall dedicate this section to the study of the relation between Higgs bundles fixed by the induced action of an anti-holomorphic involution $f : \Sigma \to \Sigma$,  and bundles with real or quaternionic structure. In particular, we shall consider the case of $SL(2,\mathbb{C})$-Higgs bundles. 

Let $E,F$ be rank $n$ holomorphic vector bundles on $\Sigma$, and $\phi : E \to F$ be an anti-linear bundle isomorphism covering $f$. The map $\phi$ can be extended to a map $\phi : \Omega^{(p,q)}(E) \to \Omega^{(p,q)}(F)$ of form-valued sections as follows. For any point $x \in \Sigma$ choose a local trivialisation of $E$ near $x$, and a trivialisation of $F$ near $f(x)$. A local section $s$ of $E$ in this trivialisation is a $\mathbb{C}^n$-valued function defined near $x$. The corresponding local section of $F$ near $f(x)$ is of the form $\phi(s) = g f^*(\overline{s})$, for some locally defined $GL(n,\mathbb{C})$-valued function $g$. The extension of $\phi$ to form-valued sections is given by setting $\phi( \omega \otimes s ) = f^*(\overline{\omega}) \otimes g f^*(\overline{s})$, where $\omega$ is a form on $\Sigma$ defined near $x$. We say that such a map $\phi : E \to F$ is holomorphic if it sends holomorphic sections to holomorphic sections, or if , letting $\overline{\partial}_E,\overline{\partial}_F$ denote the corresponding $\overline{\partial}$-operators on $E,F$, we have $\phi \circ \overline{\partial}_E = \overline{\partial}_F \circ \phi$.

Let $(E,\Phi)$ be a rank $n$ Higgs bundle satisfying the Hitchin equations and whose isomorphism class is fixed by $f$. Since $E$ must have degree $0$ we can take it to be the trivial rank $n$ bundle equipped with the constant Hermitian structure. Let $\overline{\partial}_A$ be the $\overline{\partial}$-operator defining the holomorphic structure, so $\overline{\partial}_A = \overline{\partial} + A$ for some $A \in \Omega^{0,1}( \mathfrak{gl}(n,\mathbb{C}))$. Since $(E,\Phi)$ satisfies the Hitchin equations we have that $f(E,\Phi) = ( f^*(\partial_A) , f^*(\overline{\Phi}^t) )$, where $\partial_A = \partial - \overline{A}^t$. 
Moreover, $f^*(\partial_A) = \overline{\partial} - f^*(\overline{A}^t)$ and thus supposing that $f(E,\Phi)$ is isomorphic to $(E,\Phi)$   is equivalent to the existence of an anti-linear isomorphism $\phi : E \to E^*$ covering $f$, holomorphic in the sense described above and such that $\phi \circ \Phi^t = \Phi \circ \phi$. Note that this gives a purely holomorphic interpretation of the condition for a Higgs bundle $(E,\Phi)$ to be fixed by $f$.

Suppose now that $(E,\Phi)$ is an $SL(2,\mathbb{C})$-Higgs bundle. Then we have an isomorphism $(E^* , \Phi^t) \simeq (E , -\Phi)$. In this case if the isomorphism class is fixed by $f$ we have a holomorphic anti-linear isomorphism $\phi : E \to E$ covering $f$ and such that $\phi^{-1} \circ \Phi \circ \phi = -\Phi$. Note then that $\phi^2 : E \to E$ is a linear isomorphism covering the identity, preserving the holomorphic structure and $\Phi$. As shown in \cite{N1}, it follows that $\phi^2$ preserves the corresponding flat connection $\nabla = \nabla_A + \Phi + \Phi^*$. Assuming $\nabla$ is stable then $\phi^2 = \lambda \in \mathbb{C}^*$ is constant. Rescaling by a positive constant we may assume $\lambda$ has norm $1$.

\begin{prop}
The constant $\lambda$ satisfies $\lambda = \pm 1$.
\end{prop}
\begin{proof}
Taking determinants we have a holomorphic anti-linear isomorphism covering $f$ given by ${\rm det}(\phi) : \mathbb{C} \to \mathbb{C}$. Thus ${\rm det}(\phi)(s) = \alpha f^*(\overline{s})$, for some constant $\alpha \in \mathbb{C}^*$. However, we also have that ${\rm det}(\phi) \circ {\rm det}(\phi) = \lambda^2$, hence $\lambda^2 = \alpha \overline{\alpha}$ is real and positive. But $\lambda$ has norm $1$, so $\lambda^2 = 1$.
\end{proof}

From the above analysis, if we assume $(E,\Phi)$ is stable, the bundle $E$ has either a real or quaternionic structure according to whether $\lambda = 1$ or $\lambda = -1$, respectively. We will now examine how this distinction is reflected in the spectral data for $(E,\Phi)$.

Suppose that $(E,\Phi)$ is a stable $SL(2,\mathbb{C})$-Higgs bundle with smooth spectral curve $p : S \to \Sigma$ and which is a fixed point of the action of $f$. Recall from Section \ref{sec:rank2case} that $S$ is the double cover $S= \{ (\eta, z) \in K ~|~ \eta^{2}=q(z)\}$ associated to a quadratic differential $q \in H^0(\Sigma , K^2)$ such that $f^*(\overline{q}) = q$ and $q$ has only simple zeros. Let $K^{1/2}$ be an $f$-invariant theta characteristic. Then we have $E = p_*( L \otimes p^*(K^{1/2}))$ for some $L \in {\rm Prym}(S,\Sigma)$, and $\Phi$ is obtained by pushing forward the tautological section $\eta : L \otimes p^*(K^{1/2}) \to L \otimes p^*(K^{3/2})$. Let $\tilde{f} : S \to S$ be the natural lift of $f$ to $S$ and $\sigma : S \to S$ the automorphism which exchanges sheets of the covering. Since $(E,\Phi)$ is fixed by the action of $f$ we have $\tilde{f}^*(\overline{L}) = L^*$ and since $L \in {\rm Prym}(S,\Sigma)$ we have $\sigma^*(L) = L^*$. Recall that $\tilde{f} : K \to K$ is defined by $\tilde{f}(\omega) = f^*(\overline{\omega})$.

\begin{lem}\label{lem:gamma}
There exists a holomorphic anti-linear isomorphism $\gamma : K^{1/2} \to K^{1/2}$ covering $f$ such that $\gamma \circ \gamma = \epsilon_1 = \pm 1$ and $\gamma \otimes \gamma = \tilde{f}$. If $f$ has fixed points then $\epsilon_1 = 1$.
\end{lem}
\begin{proof}
Since $f^*(\overline{K}^{1/2}) \simeq K^{1/2}$, we have that there exists a holomorphic anti-linear isomorphism $\gamma : K^{1/2} \to K^{1/2}$ covering $f$. Then $\gamma \circ \gamma = \beta \in \mathbb{C}^*$ for some constant $\beta$. Rescaling by a positive constant we can assume $\beta$ has norm $1$. We also have $(\gamma \otimes \gamma)(\omega) = \alpha f^*(\overline{\omega})$ for some constant $\alpha \in \mathbb{C}^*$. Then $\alpha \overline{\alpha} = \beta^2$, so that $\alpha \overline{\alpha} = \beta^2 = 1$. Choose $u \in \mathbb{C}^*$ with $u^2 = \alpha$ and replace $\gamma$ by $u \gamma$. Then $\gamma \otimes \gamma = \tilde{f}$ and $\gamma \circ \gamma = \beta = \epsilon_1$, where $\beta^2 = 1$. Finally if $f$ has fixed points then we must have $\beta = 1$ since otherwise we would obtain a quaternionic structure on a rank one complex vector space.
\end{proof}

Let $\eta \in \Gamma(S,p^*(K))$ denote the tautological section. We have $\tilde{f}^*(\eta) = \eta$ and $\sigma^*(\eta) = -\eta$. Seting $\tau = \tilde{f} \sigma$, the map  $\tau : S \to S$ is an anti-holomorphic involution and $\tau^*(\eta) = -\eta$. Moreover we also have $\tau^*(\overline{L}) = L$. Suppose now that $\tilde{\tau} : L \to L$ is a lift of $\tau$ to a holomorphic anti-linear map satisfying $\tilde{\tau} \circ \tilde{\tau} = \epsilon_2$, where $\epsilon_2 = \pm 1$. Let $\gamma : K^{1/2} \to K^{1/2}$ be as in Lemma \ref{lem:gamma}. Then since $E = p_*(L \otimes p^*(K^{1/2})$, it is clear that $\tilde{\tau} \otimes p^*(\gamma)$ pushes-forward to a holomorphic anti-linear involution $\phi : E \to E$ covering $f$, satisfying $\phi^{-1} \circ \Phi \circ \phi = -\Phi$ and $\phi^2 = \epsilon$, where $\epsilon = \epsilon_1 \epsilon_2$. Thus $E$ is real or quaternionic according to whether $\epsilon = 1$ or $\epsilon = -1$.

\begin{prop}\label{prop:taufixed}
Suppose $q$ has a zero fixed by $f$. Then there exists a lift $\tilde{\tau} : L \to L$ of $\tau$ with $\tilde{\tau} \circ \tilde{\tau} = 1$. In this case $E$ has a real structure.
\end{prop}
\begin{proof}
Since $\tau = \tilde{f} \sigma$ it follows that $\tau$ has fixed points. Let $x \in S$ be a fixed point of $\tau$. Since $L$ is a holomorphic line bundle, then  there is associated to $L$ a corresponding unitary homomorphism $\rho : \pi_1(S,x) \to U(1)$, where $x \in S$. Let $\tilde{S} \to S$ be the universal cover viewed as a principal $\pi_1(S,x)$-bundle. Then $L$ may be identified as the quotient of $\tilde{S} \times \mathbb{C}$ by the relation $(p g , s) \sim (p , \rho(g) s)$. Let $\hat{\tau} : \pi_1(S,x) \to \pi_1(S,x)$ be the automorphism of $\pi_1(S,x)$ constructed as in Section \ref{sec:ac.fund}. Then since $x$ is a fixed point, $\hat{\tau}$ is an involution. Moreover since $\tau^*(\overline{L}) = L$ we have $\rho \circ \hat{\tau} = \overline{\rho}$. We can then lift $\tau$ to an involution $\tau' : \tilde{S} \to \tilde{S}$ which satisfies $\tau'(pg) = \tau'(p) \hat{\tau}(g)$. Now define $\tilde{\tau} : L \to L$ by $\tilde{\tau}(p , s) = (\tau'(p) , \overline{s})$. This is well-defined since $\rho \circ \hat{\tau} = \overline{\rho}$, and is an involution since $\tau'$ is an involution. Finally note that $\epsilon_1$ in Lemma \ref{lem:gamma} equals $1$ since $f$ has fixed points.
\end{proof}
\begin{prop}\label{prop:taunofix}
Supppose $q$ has no zeros fixed by $f$. Let $x \in S$ and choose a path $l$ from $x$ to $\tau(x)$. Let $\mu$ be the loop $\mu = l.\tau(l)$ and let $\rho : \pi_1(S,x) \to U(1)$ be the flat unitary structure associated to the holomorphic line bundle $L$. Then there exists a lift $\tilde{\tau} : L \to L$ of $\tau$ with $\tilde{\tau} \circ \tilde{\tau} = \epsilon = \rho(\mu)$.
\end{prop}
\begin{proof}
The proof is nearly identical to that of Proposition \ref{prop:taufixed}, except now the lift $\tau' : \tilde{S} \to \tilde{S}$ satisfies $\tau'(\tau'(p)) = p \mu^{-1}$, hence the lift $\tilde{\tau} : L \to L$ given by $\tilde{\tau}(p,s) = (\tau(p) , \overline{s})$ squares to $\rho(\mu)^{-1}$. Note also that since $\rho \circ \hat{\tau} = \overline{\rho}$ and $\hat{\tau}(\mu) = \mu$, we have that $\rho(\mu) = \rho(\mu)^{-1} = \pm 1$. 
\end{proof}

 
\section{Langlands duality for $\CL_{G_{c}}$}\label{sec:lang}

Let $^{L}G_c$ denote the Langlands dual group of $G_c$. The bases $\CA_{G_c}$, $\CA_{^{L}G_c}$ of the Hitchin fibrations $h : \CM_{G_c} \to \CA_{G_c}$ and $h : \CM_{^{L}G_c} \to \CA_{^{L}G_c}$ can naturally be identified so that $\CM_{G_c},\CM_{^{L}G_c}$ are torus fibrations over a common base $\CA_{G_c} \simeq \CA_{^{L}G_c}$. Langlands duality is then interpreted as the statement that the moduli spaces $\CM_{G_c},\CM_{^{L}G_c}$ are duals in the sense of mirror symmetry. One aspect of this duality is that the connected components of the non-singular fibres of $\CM_{G_c}$ and $\CM_{^{L}G_c}$ should be dual abelian varieties.

From \cite[Section 12.4]{Kap}, under Langlands duality we know that  $(A,B,A)$-branes in $\CM_{G_{c}}$ map to $(A,B,A)$-branes in the moduli space $\CM_{^{L}G_{c}}$ of $^{L}G_{c}$-Higgs bundles for $^{L}G_{c}$. Informally, the mapping of branes can be thought of as a fibrewise Fourier-Mukai transform. Hence we know that $\CL_{G_{c}}\subset \CM_{G_{c}}$ maps under the duality to an $(A,B,A)$-brane $^{L}\CL_{G_{c}} \subset \CM_{^{L}G_{c}}$. We claim that the dual $(A,B,A)$-brane $^{L}\CL_{G_c} \subset \mathcal{M}_{^{L}G_c}$ coincides with the $(A,B,A)$-brane $\mathcal{L}_{^{L}G_c}$.

Consider first the case $G_c = \, ^{L} G_c = GL(n,\mathbb{C})$ and restrict attention to the moduli space of degree $0$ Higgs bundles $\CM_{GL(n,\mathbb{C})} = \CM_{GL(n,\mathbb{C})}^{{\rm Higgs},0}$. Let $p : S \to \Sigma$ be the spectral curve corresponding to a generic point $a \in \CA_{GL(n,\mathbb{C})}$. After fixing a choice of spin structure the fibre of $\CM_{GL(n,\mathbb{C})}$ is given by the Jacobian $Jac(S)$, and the self-duality of $\CM_{GL(n,\mathbb{C})}$ is reflected in the self-duality of $Jac(S)$. Now let $L \subset \CA_{GL(n,\mathbb{C})}$ be the subspace of $\CA_{GL(n,\mathbb{C})}$ fixed by the induced action of $f$, so that if $a \in L$ then $f$ lifts to an involution $\tilde{f} : S \to S$. Then $\CL_{GL(n,\mathbb{C})}$ fibres over $L$ with fibre  corresponding to the subspace of $Jac(S)$ fixed by the map $\iota_0 : Jac(S) \to Jac(S)$ as in Definition \ref{def:iota0}. Let $B \subset Jac(S)$ be the fixed point set of $\iota_0$. 

The self-duality of $\CL_{GL(n,\mathbb{C})}$ is reflected in the fact that $B$ is a Lagrangian submanifold of $Jac(S)$, where the symplectic structure on $Jac(S)$ is given by the standard principal polarization. This follows since the map $H^1(S,\mathbb{R}) \to H^1(S,\mathbb{R})$ corresponding to $\iota_0$ is given by $x \mapsto -\tilde{f}^*(x)$, which is anti-symplectic.

To understand the above claim we need a brief digression into Fourier-Mukai duality. Let $A,\hat{A}$ be dual abelian varieties. There is a natural dual pairing of $H_1(A,\mathbb{R})$ and $H_1(\hat{A},\mathbb{R})$. Given a subspace $V \subseteq H_1(A,\mathbb{R})$ there is a corresponding subspace $\hat{V} \subseteq H_1(\hat{A},\mathbb{R})$, namely the annihilator $\hat{V} = V^\perp$ of $V$. Under the cohomological Fourier-Mukai transform \cite{huy} $H^*(A , \mathbb{R}) \to H^*(\hat{A},\mathbb{R})$, the Poincar\'{e} dual cohomology classes $\eta_V \in H^*(A,\mathbb{R})$ and  $\eta_{\hat{V}} \in H^*(\hat{A},\mathbb{R})$ correspond to one another (up to a sign factor). Based on this we say that closed subtori $B \subset A$ and  $\hat{B} \subset \hat{A}$ are Fourier-Mukai dual if the corresponding subspaces $V \subset H_1(A,\mathbb{R})$ and  $\hat{V} \subset H_1(\hat{A},\mathbb{R})$ are annihilators of one another.
When $A$ is self-dual we can use a principal polarization to identify $H_1(A,\mathbb{R})$ with $H_1(\hat{A},\mathbb{R})$ and we find that a subspace $V \subseteq H_1(A,\mathbb{R})$ is self-dual if and only if it is a Lagrangian subspace. This explains our claim that $\mathcal{L}_{GL(n,\mathbb{C})}$ is a self-dual brane.

Consider now the case $G_c = SL(n,\mathbb{C})$, for which  $^{L}G_c = PGL(n,\mathbb{C})$. Clearly we have an identification of the bases $\CA_{SL(n,\mathbb{C})} \simeq \CA_{PGL(n,\mathbb{C})}$. Let $a$ be a generic point of the base with corresponding spectral curve $p : S \to \Sigma$. To see the duality of the fibres we note that the exact sequence
\begin{equation*}
\xymatrix{
1 \ar[r]& {\rm Prym}(S,\Sigma) \ar[r]& {\rm Jac}(S) \ar[r]^{Nm}& {\rm Jac}(\Sigma) \ar[r]& 1}
\end{equation*}
dualises to
\begin{equation*}\xymatrix{
1 \ar[r]& {\rm Jac}(\Sigma) \ar[r]^{p^*}& {\rm Jac}(S) \ar[r]& \hat{{\rm Prym}}(S,\Sigma) \ar[r]& 1.}
\end{equation*}
In particular this shows that the dual of ${\rm Prym}(S,\Sigma)$ is ${\rm Jac}(S)/ p^*({\rm Jac}(\Sigma))$ which is easily seen to describe the spectral data for $PGL(n,\mathbb{C})$-Higgs bundles.

Let $a$ be a point in the base which is fixed by the action of $f$, so that $f$ lifts to $\tilde{f} : S \to S$. As before, denote by  $B \subset {\rm Jac}(S)$   the fixed point set of $\iota_0$. Then the fibre of $\CL_{SL(n,\mathbb{C})}$ over $a$ is $B \cap {\rm Prym}(S,\Sigma)$, and the fibre of $\CL_{PGL(n,\mathbb{C})}$ is the image of $B$ in ${\rm Jac}(S)/ p^*({\rm Jac}(\Sigma))$. The components of these fibres are then seen to be dual in the sense described above.


\section{Higgs bundles and 3-manifolds}\label{sec:3man}



One can construct a canonical 3-manifold $M$ from the compact Riemann surface $\Sigma$, together with an anti-holomorphic involution $f$ as introduced in Section \ref{sec:anti-inv}, which relates the $(A,B,A)$-brane $\CL_{G_{c}}$, and representations on $M$ and on the boundary $\partial M$. For this, we consider the product 
$\overline{\Sigma}=\Sigma \times [-1,1].$
On $\overline{\Sigma}$ there is a natural involution  
\begin{eqnarray*}
\sigma: (x,t)\mapsto (f(x),-t),
\end{eqnarray*}
which is orientation preserving, and   a product action  (e.g., \cite[Section 1]{km}). 
The quotient  $M=\bar \Sigma /\sigma$ is a 3-dimensional manifold with boundary $\partial M =\Sigma$.
From $M$ we obtain a distinguished subspace of representations of $\pi_1(\Sigma)$, namely those representations which viewed as flat connections on $\Sigma$, extend to flat connections over $M$. 

\begin{proposition}The representations of $\Sigma$ which extend to $M$ belong to the $(A,B,A)$-brane introduced in this paper. \end{proposition}

\begin{proof} Let $i : \Sigma \to \overline{\Sigma}$ be the inclusion $i(x) = (x,0)$. Then clearly $i \circ f = \sigma \circ i$. Fix a point $x_0 \in \Sigma$ and choose a path $\gamma$ in $\Sigma$ from $x_0$ to $f(x_0)$. Recall from Section \ref{sec:ac.fund} that the automorphism $\hat{f} : \pi_1(\Sigma , x_0) \to \pi_1(\Sigma,x_0)$ is given by $\hat{f}(u) = \gamma . f(u) . \gamma^{-1}$. Let $m_0 \in M$ be the point of $M$ corresponding to $(x_0,0)$ and $\tau$ the image of $i(\gamma)$ in $M$, which is a loop at $m_0$. Let $j : \Sigma \to M$ be the composition of $i$ with the projection $\overline{\Sigma} \to M$. We then have a commutative diagram 
\begin{equation*}\xymatrix{
\pi_1(\Sigma , x_0 ) \ar[r]^{j_*} \ar[d]^{\hat{f}} & \pi_1(M,m_0) \ar[d]^{{\rm Ad}_\tau} \\
\pi_1(\Sigma , x_0 ) \ar[r]^{j_*} & \pi_1(M,m_0)
}
\end{equation*} 
where ${\rm Ad}_\tau$ denotes conjugation by $\tau$. For $\rho : \pi_1(M,m_0) \to G_c$ a representation of $M$,  we  have $\rho \circ j_* \circ \hat{f} = {\rm Ad}_{\rho(\tau)} \circ \rho \circ j_*$. Hence $\rho \circ j_*$ defines a fixed point of the action of $f$ on ${\rm Rep}^+(\pi_1(\Sigma) , G_c)$ as required.
\end{proof}
 The study of this correspondence, as well as the further characterisation of the $(A,B,A)$-branes $\CL_{G_{c}}$ shall appear in the companion paper \cite{BS13}.
 
\newpage

 \appendix

\section{Anti-holomorphic involutions on hyperelliptic surfaces}\label{sec:hyper}

 As seen in \cite[Section 6]{GH81}, in the case of hyperelliptic curves more information can be deduced concerning the topological invariants $(n,a)$. We shall see here some examples. 

\subsection{Classification}\label{sec:clas}

We shall begin by classifying pairs $(\Sigma , f)$ where $\Sigma$ is a hyperelliptic curve of genus $g \ge 2$ and $f : \Sigma \to \Sigma$ an anti-holomorphic involution.
\begin{definition}\label{def:conj}
Let $\tau : \mathbb{P}^1 \to \mathbb{P}^1$ be the anti-holomorphic involution which in projective coordinates $[z_1,z_2]$ is given by $\tau( [z_1,z_2] ) = [\overline{z_1} , \overline{z_2}]$. Let $\alpha : \mathbb{P}^1 \to \mathbb{P}^1$ be given by $\alpha( [z_1,z_2] ) = [ z_2 , -z_1]$ so that $\alpha \tau = \tau \alpha$ is the antipodal map.\end{definition}

\begin{prop}
There exists a non-constant meromorphic function $z : \Sigma \to \mathbb{P}^1$ with polar divisor $z^{-1}(\infty)$ of degree $2$ such that either $z( f(w) ) = \tau(z(w))$ or $z(f(w)) = \alpha \tau( z(w))$. In addition, we may choose $z$ such that $\infty = [0,1] \in \mathbb{P}^1$ is not a branch point.
\end{prop}
\begin{proof}
Since $\Sigma$ is hyperelliptic there exists a meromorphic function $z' : \Sigma \to \mathbb{P}^1$ which has two simple  poles, or a single pole of order $2$. Any other meromorphic function with this property is obtained from $z'$ by a M\"obius transformation. In particular this applies to $\tau \circ z' \circ f$, so there exists a matrix $M \in GL(2,\mathbb{C})$ such that $\tau \circ z' \circ f = M \circ z'$, where $M$ acts on $\mathbb{P}^1$ as a M\"obius transformation. Note that since $\tau$ and $f$ are involutions we have $\overline{M} M = \pm I$. Choose a matrix $A \in GL(2,\mathbb{C})$ and set $z = A \circ z'$. We have $\tau \circ z \circ f = \overline{A} M A^{-1} z$.

If $\overline{M}M = I$ then we can choose $A \in GL(2,\mathbb{C})$ such that $\overline{A}^{-1}A = M$. Thus $z \circ f = \tau \circ z$ as required. By composing with a transformation in $GL(2,\mathbb{R})$ we can ensure that $\infty$ is not a branch point of $z$.
If $\overline{M}M = -I$ then we can choose $A \in GL(2,\mathbb{C})$ such that $\overline{A}^{-1} J A = M$, where $J$ is the linear transformation $J(z_1,z_2) = (z_2 , -z_1)$. Thus $z \circ f = \tau \alpha \circ z$ as required. The linear transformations in $GL(2,\mathbb{C})$ commuting with the antipodal map $(z_1,z_2) \to (\overline{z_2} , -\overline{z_1})$ form  the group $SU(2)$. Composing with an element of $SU(2)$ we can ensure that $\infty$ is not a branch point of $z$.
\end{proof}

We have established that the anti-holomorphic involution $ f$ on $\Sigma$ covers an anti-holomorphic involution on $\mathbb{P}^1$ which is either the conjugation map $\tau$ or the antipodal map $\alpha \tau$. Let $z : \Sigma \to \mathbb{P}^1$ be the meromorphic function as above. Then $z$ exhibits $\Sigma$ as a branched double cover of $\mathbb{P}^1$. We may choose $z$ so that $\infty$ is not a branch point, and let $P_1, \dots , P_{2g+2} \in \Sigma$ be the branch points with $z(P_1), \dots , z(P_{2g+2}) \in \mathbb{C} \subset \mathbb{P}^1$ their images in $\mathbb{P}^1$. On $\Sigma$ there is a meromorphic function $w$ satisfying 
\begin{eqnarray*}w^2 =  \prod_{j=1}^{2g+2}(z - z(P_j)):=p(z)\end{eqnarray*} 
 and a holomorphic involution $\sigma : \Sigma \to \Sigma$, the hyperelliptic involution,  defined by exchanging the sheets of the branched covering $\Sigma \to \mathbb{P}^1$. Thus $z \circ \sigma = z$ and $w \circ \sigma = -w$.

\begin{proposition}
If the anti-holomorphic involution $f$ covers the conjugation map $\tau$ from Definition \ref{def:conj}, then it is given by either $f': (w,z)\mapsto (\bar w,\bar z)$ or by $\sigma \circ f'$.
\end{proposition}\begin{proof}
 For simplicity we  shall write $\overline{z}$ in place of $\tau(z)$. Since $f$ sends branch points of $z$ to branch points, the set of images $z(P_1), \dots , z(P_{2g+2})$ of branch points is left invariant under the conjugation map $\tau$. Therefore each $z(P_j)$ is either real or occurs in a conjugate pair. The polynomial $p(z) = \prod_{j=1}^{2g+2}(z - z(P_j))$ thus satisfies $p( \overline{z} ) = \overline{ p(z) }$. In particular we obtain an anti-holomorphic involution $f' : \Sigma \to \Sigma$ by sending a pair $(w,z)$ such that $w^2 = p(z)$ to the corresponding pair $(\overline{w} , \overline{z})$. Clearly $f'$ covers $\tau$. The only other anti-holomorphic involution covering $\tau$ is given by sending a pair $(w,z)$ to $(-\overline{w} , \overline{z})$ and this is just $\sigma \circ f'$. Thus $f$ is one of $f'$ or $\sigma \circ f'$.\end{proof}

\begin{prop}
Let $2k$ be the number of real roots of $p(z)$. And consider the topological invariants $(n_{\Sigma},a_{\Sigma})$ associated to the pair $(\Sigma , f')$. If $k=0$ then $(n_{\Sigma},a_{\Sigma}) = (1,0)$ if $g$ is even, and $n_{\Sigma}=2$ if $g$ is odd. If $0 < k < g+1$ then $(n_{\Sigma},a_{\Sigma}) = (k,1)$, and if $k=g+1$ then $(n_{\Sigma},a_{\Sigma}) = (g+1,0)$.
\end{prop}
\begin{proof}
Consider the zeros $z(P_1), \dots , z(P_{2g+2})$ of $p(z)$ in the complex plane. Draw branch cuts between conjugate pairs of roots and joining adjacent pairs of real roots. From this data the invariants $(n_{\Sigma},a_{\Sigma})$ can easily be determined.
\end{proof}

In a similar manner   the topological invariants of the involution $\sigma \circ f'$ can be found.
\begin{prop}
Let $2k$ be the number of real roots of $p(z)$. And consider the topological invariants $(n_{\Sigma},a_{\Sigma})$ associated to the pair $(\Sigma , \sigma \circ f')$. If $0 \le k < g+1$ then $(n_{\Sigma},a_{\Sigma}) = (k,1)$, and if $k=g+1$ then $(n_{\Sigma},a_{\Sigma}) = (g+1,0)$.
\end{prop}

We shall now  consider involutions covering the antipodal map $\alpha \tau$ as in Definition \ref{def:conj}. In this case the antipodal map must permute the zeros of $p(z)$, so the zeros occur in antipodal pairs. For $\lambda := \prod_{j=1}^{2g+2} \overline{z(P_j)}$, it follows that $\lambda$ is of the form $\lambda = \mu \left( \frac{-1}{\overline{\mu}} \right) = -\mu^2/|\mu|^2$, for some non-zero $\mu \in \mathbb{C}$. Define $c = i \mu / |\mu|$ so that $c^2 = \lambda$ and $c\overline{c} = 1$. It is clear that if $w^2 = p(z)$ then $\left( \frac{\overline{w}}{c \overline{z}^{g+1}} \right)^2 = p( -1/\overline{z})$, so we obtain an anti-holomorphic map $h : \Sigma \to \Sigma$ such that $z \circ h = z$ by setting $h(w,z) = ( \left( \frac{\overline{w}}{c \overline{z}^{g+1}} \right) , -1/\overline{z})$. Note that $h$ is an involution if and only if $g$ is odd. We have thus established:
\begin{prop}
If $g$ is even, then there are no anti-holomorphic involutions on $\Sigma$ covering the antipodal map on $\mathbb{P}^1$. If $g$ is odd then there are precisely two such invoutions, $h$ and $\sigma \circ h$. The topological invariants of  $(\Sigma , h)$ and $(\Sigma , \sigma \circ h)$ are $(n_{\Sigma},a_{\Sigma}) = (0,1)$.
\end{prop}


\subsection{Genus 2 case}

Let $z,w$ be the meromorphic functions on $\Sigma$ as in Section \ref{sec:clas}. Then $H^0(\Sigma , K)$ has dimension $g$ and is spanned by $z^j dz/w$ for $j=0,1, \dots , g-1$. A divisor is in the linear system $\mathbb{P}( H^0(\Sigma , K))$ if and only if it is of the form $D = z^{-1}(a_1) + \dots + z^{-1}(a_{g-1})$, where $a_1,\dots , a_{g-1}$ are any $(g-1)$ points on $\mathbb{P}^1$. 

For a hyperelliptic surface of genus $g = 2$ the anti-holomorphic involution $f$ must cover the conjugation map on $\mathbb{P}^1$. Moreover, since the natural map $H^0(\Sigma , K) \to H^0(\Sigma , K^2)$ is surjective, in this case we can describe the linear system of quadratic differentials. These all have the form $D = z^{-1}(a_1) + z^{-1}(a_2)$ for two points $a_1,a_2 \in \mathbb{P}^1$. The quadratic differentials $q$ with divisor $(q) = D$ have simple zeros provided that $a_1 \neq a_2$ and that $a_1,a_2$ are not zeros of $p(z)$.

Let $a_1,a_2 \in \mathbb{P}^1$ be distinct points that are not zeros of $p(z)$, and let $q$ be any non-zero quadratic differential with $(q) = D = z^{-1}(a_1) + z^{-1}(a_2)$. Then $f^*(\overline{q})$ has divisor $z^{-1}(\overline{a_1}) + z^{-1}(\overline{a_2})$. Thus to obtain a quadratic differential $q$ with $f^*(\overline{q}) = q$, the points $a_1,a_2$ must be real or conjugates. For such a pair $a_1,a_2 \in \mathbb{P}^1$, the quadratic differentials with divisor $D = z^{-1}(a_1) + z^{-1}(a_2)$ determine a complex $1$-dimensional space of $H^0(\Sigma , K^2)$ invariant under the induced action of $f$. Thus we can find always find a non-zero quadratic differential $q$ with $(q) = D$ and $f^*(\overline{q}) = q$. Moreover $q$ is unique up to the action of $\mathbb{R}^*$.

Consider the spectral curve $S \to \Sigma$ given by the characteristic equation $\lambda^2 = q$ and with induced involution $\tilde{f}(\lambda , x) = (f^*(\overline{\lambda}) , f(x) )$, where $x \in \Sigma$ and $\lambda^2 = q(x)$. Replacing $q$ by any positive multiple of itself will simply require a corresponding rescaling of $\lambda$, so the pair $(S , \tilde{f})$ is essentially independent of such rescalings. On the other hand replacing $q$ by $-q$ will leave $S$ unchanged but replace $\tilde{f}$ by $\pi \circ \tilde{f}$, where $\pi : S \to S$ is the involution which exchanges sheets of $S \to \Sigma$.

To summarise, we choose distinct points $a_1,a_2 \in \mathbb{P}^1$ which are either both real or complex conjugates. Up to positive rescalings there are two non-zero quadratic differentials $q,-q$ with $(q) = (-q) = z^{-1}(a_1) + z^{-1}(a_2)$. In fact, we may take $q$ to be
\begin{equation*}
q = (z-a_1)(z-a_2)(dz)^2 / w^2,
\end{equation*}
if $a_1,a_2$ are finite. If say $a_2 = \infty$ then we omit the factor $(z-a_2)$. From here it is possible to calculate the invariants $(n_+,u/2)$ associated to $q$ as in Section \ref{sec:rank2case}, where $n_+$ is the number of fixed components of $f$ on which $q$ is non-vanishing and positive and $u$ is the number of zeros of $q$ which are fixed by $f$.

\end{document}